\documentclass[11pt,reqno]{amsart}
\usepackage[margin=1in]{geometry}

\newcommand{\fr}{\mathfrak}

\newcommand{\bb}{\mathbb}

\newcommand{\RR}{\mathbb{R}}
\newcommand{\CC}{\mathbb{C}}

\newcommand{\End}{\mathrm{End}}
\newcommand{\Aut}{\mathrm{Aut}}
\newcommand{\an}{\mathrm{an}}
\newcommand{\alg}{\mathrm{alg}}

\newcommand{\loc}{\mathrm{loc}}
\newcommand{\vol}{\mathrm{vol}}
\newcommand{\HYM}{\mathrm{HYM}}
\newcommand{\Gr}{\mathrm{Gr}}
\newcommand{\hns}{\mathrm{hns}}
\newcommand{\hn}{\mathrm{hn}}

\newcommand{\rank}{\mathrm{rank}}

\newtheorem*{thm*}{Theorem}

\newtheorem{thm}{Theorem}[section]

\newtheorem{prop}[thm]{Proposition}
\newtheorem{lem}[thm]{Lemma}

\newtheorem{cor}[thm]{Corollary}

\theoremstyle{definition}
\newtheorem{defn}[thm]{Definition}

\newtheorem{rem}[thm]{Remark}

\usepackage{graphicx}
\usepackage{amssymb}
\usepackage{epstopdf}
\usepackage{enumerate}
\usepackage{hyperref}

\numberwithin{equation}{section}

\title{Limiting behavior of Donaldson's heat flow on non-K\"{a}hler surfaces}

\author{Jacob McNamara}
\address{Department of Mathematics, Harvard University, Cambridge, MA 02138}
\email{jmcnamara@college.harvard.edu}

\author{Yifei Zhao}
\address{Department of Mathematics, Columbia University, New York, NY 10027}
\email{yz2427@columbia.edu}
\thanks{Both authors are supported in part by NSF grants DMS-07-57372 and DMS-12-66033}

\date{March 24, 2014}

\begin{document}
\maketitle

\begin{abstract}
Let $X$ be a compact Hermitian surface, and $g$ be any fixed Gauduchon metric on $X$. Let $E$ be an Hermitian holomorphic vector bundle over $X$. On the bundle $E$, Donaldson's heat flow is gauge equivalent to a flow of holomorphic structures. We prove that this flow converges, in the sense of Uhlenbeck, to the double dual of the graded sheaf associated to the $g$-Harder-Narasimhan-Seshadri filtration of $X$. This result generalizes a convergence theorem of Daskalopoulos and Wentworth to non-K\"{a}hler setting.
\end{abstract}

\setcounter{section}{-1}
\section{Introduction}
The interplay between existence of canonical metrics and algebraic stability conditions has become one of the major topics in complex differential geometry. Given an indecomposable holomorphic vector bundle $(E,\bar{\partial})$ over a compact K\"{a}hler manifold $(X,g)$, it is known that $E$ admits an Hermitian-Einstein metric if and only if $E$ is stable in the sense of Mumford-Takemoto. This result was proved for complex curves by Narasimhan and Seshadri \cite{Nar}, and then in full generality by Donaldson \cite{Don} and Uhlenbeck-Yau \cite{Uhl-Yau}, where it is now known as the Donaldson-Uhlenbeck-Yau theorem. In particular, Uhlenbeck and Yau showed that the existence of an Hermitian-Einstein metric is equivalent to a $C^0$-estimate, whose violation would produce a destabilizing subsheaf, while Donaldson constructed a non-linear flow of self-adjoint endomorphisms of the bundle which would converge to the Hermitian-Einstein metric if the bundle were Mumford-Takemoto stable.

This stability condition under consideration is defined using the degree/rank ratio of a coherent sheaf $E$. When $X$ is non-K\"{a}hler, the notion of degree can be defined by fixing an Hermitian metric $g$ on $X$ whose associated $(1,1)$-form $\omega$ satisfies $\partial\bar{\partial}(\omega^{n-1})=0$. Such a metric is called a \emph{Gauduchon metric}, and it exists in the conformal class of every Hermitian metric (cf. \cite{Gau}). By a theorem of L\"{u}bke and Teleman \cite{Lub}, the stability condition corresponding to a Gauduchon metric is still equivalent to the existence of Hermitian-Einstein metrics. However, the heat flow solution in this generality is only recently given by Jacob \cite{Jac}.

In the case where $E$ is not assumed to be stable, one can still ask about the limiting behavior of Donaldson's heat flow. Indeed, when $X$ is K\"{a}hler, Donaldson's heat flow is gauge equivalent to the Yang-Mills flow, which is a flow of integrable, unitary connections on $E$. In \cite{Das} and \cite{Das-Wen}, Daskalopoulos and Wentworth discovered a relation between the limit of the Yang-Mills flow and the Harder-Narasimhan filtration, first over Riemann surfaces and then over compact K\"{a}hler surfaces. In the dimension 2 case, the convergence has to take into account ``bubbling phenomena." More precisely, along the Yang-Mills flow $D_t$ with $D_0''=\bar{\partial}$, we have uniform control over the functionals $\|F_{D_t}\|_{L^2}$ and $\|\Lambda F_{D_t}\|_{L^{\infty}}$ (where $\Lambda$ is the adjoint of the Lefschetz operator $L=\omega\wedge$), which implies certain weak subsequential convergence of $D_t$ away from a finite subset of $X$. The main result of \cite{Das} is that such a limit can be identified, away from a finite subset of $X$, with the reflexified graded object $\Gr^{\hns}(E,\bar{\partial})^{**}$ of the Harder-Narasimhan-Seshadri filtration of $(E,\bar{\partial})$.

Inspired in part by Jacob's work \cite{Jac}, we speculated that Donaldson's heat flow should exhibit similar behavior over non-K\"{a}hler manifolds. The first instances of compact non-K\"{a}hler manifolds occur in dimension 2. Here, Donaldson's heat flow is no longer gauge equivalent to the Yang-Mills flow, but rather to the following flow of integrable, unitary connections:
\begin{equation}
\label{eq-uym}
\frac{d}{dt}A_t=\frac{i}{2}\left(D''_t\Lambda F_{D_t}-D'_t\Lambda F_{D_t}\right),\quad\forall t\ge 0
\end{equation}
In this setting, we generalize the convergence result of \cite{Das-Wen}:
\begin{thm*}[main theorem]
Let $(E,\bar{\partial})$ be an Hermitian holomorphic vector bundle over a compact Gauduchon surface $(X,g)$. Let $D_t$ be the solution to \eqref{eq-uym} with initial condition $D_0''=\bar{\partial}$. Given any sequence $t_j\rightarrow\infty$, there exists a connection $D_{\infty}$ on an Hermitian vector bundle $E_{\infty}$, and a finite set $Z^{\an}\subset X$ such that
\begin{enumerate}[(i)]
	\item $(E_{\infty},D_{\infty}'')$ is holomorphically isomorphic to $\Gr^{\hns}(E,D_0'')^{**}$;
	
	\item $E$ and $E_{\infty}$ are identified outside $Z^{\an}$ via $W^{2,p}_{\loc}$-isometries for all $p$;
	
	\item Via the isometries in (ii), and after passing to a subsequence, $D_j\rightarrow D_{\infty}$ in $L^2_{\loc}$ away from $Z^{\an}$.
\end{enumerate}
\end{thm*}

\noindent
The proof of this theorem is largely built upon Daskalopoulos and Wentworth's argument in \cite{Das-Wen}. More precisely, they proved that
\begin{enumerate}[(i)]
	\item Along the Yang-Mills flow on a compact K\"{a}hler surface, the Hermitian Yang-Mills functional $\HYM(D_t):=\|\Lambda F_{D_t}\|_{L^2}$ converges to its absolute lower bound $\HYM(\vec{\mu}_0):=2\pi\sum_{i}(\vec\mu_0)_i^2$, where $\vec{\mu}_0$ is the Harder-Narasimhan type of $(E,D''_0)$\footnote{Here and throughout the paper, we assume the normalization $\vol(X)=2\pi$. This normalization has the advantage that the Hermitian-Einstein constant of a vector bundle agrees with its slope, i.e. if $i\Lambda F_D=\mu(E)\mathbb{I}_E$, then $\mu(E)=\deg(E)/\rank(E)$.} (\cite[Prop. 2.26 \& Thm. 4.1]{Das-Wen}). This step requires showing that $\lim_{t\rightarrow\infty}\HYM(D_t)=\HYM(D_{\infty})$ for some Uhlenbeck limit $(E_{\infty},D_{\infty})$, and then identifying the Harder-Narasimhan type of $(E_{\infty},D_{\infty}'')$ with that of $(E,D_0'')$;
	\item Independently of the flow equation, the conclusion of the main theorem holds for any sequence of integrable connections $D_j$ that are equivalent to $D_0$ by complex gauge transformations, and satisfying $\HYM(D_j)\rightarrow\HYM(\vec{\mu}_0)$ (\cite[Thm. 5.1]{Das-Wen}).
\end{enumerate}
Part (ii) can be directly adapted to our case, but for part (i), new estimates have to be made for the flow equation \eqref{eq-uym}.

In \S1, we introduce the flow equation \eqref{eq-uym}, and present the crucial estimate for $\|F_{D_t}\|_{L^2}$ (Prop. \ref{prop-bound-f}). This estimate not only relies on the fact that $g$ is a Gauduchon metric, but also exploits the dimension of $X$ (Lem. \ref{lem-surface-torsion}); these techniques are not featured in the K\"{a}hler case. Moreover, many estimates in the K\"{a}hler case are proved using the maximum principle for the Hodge Laplacian $\Box=d^*d+dd^*$; in the Gauduchon case, we often find it more efficient to use the operator $P=i\Lambda\bar{\partial}\partial$ introduced in \cite{Lub}. From the estimate on $\|F_{D_t}\|_{L^2}$, it also follows that $\|D_{t_j}\Lambda F_{D_{t_j}}\|_{L^2}\rightarrow 0$ along some sequence $t_j\rightarrow \infty$. This will allow us to take a subsequential Uhlenbeck limit along $t_j\rightarrow\infty$ in \S2, and obtain a weak limit connection $D_{\infty}$, satisfying $D_{\infty}\Lambda F_{D_{\infty}}=0$ in a weak sense. The rest of \S2 is devoted to the regularity and removability of singularity of this equation. Indeed, this equation is a slightly more general form of the Hermitian-Einstein equation, and agrees with the Yang-Mills equation to the top order. The methods of Yang-Mills equation will apply. In \S3, we put these ingredients together, and use the results of \cite{Das-Wen} to prove the main theorem. In the course of the proof, we generalize another result in \cite{Das-Wen} to Gauduchon surfaces:
\begin{thm*}
Let $(E,\bar{\partial})$ be a holomorphic vector bundle over a compact Gauduchon surface $(X,g)$. Given any $\delta>0$ and any $1\le p<\infty$, there is an $L^p$-$\delta$-approximate critical Hermitian structure on $E$.
\end{thm*}

\noindent
\emph{Acknowledgement.} First and foremost, both authors would like to thank their supervisor D. H. Phong, for all his advice and support. The second author would also like to thank T. Collins for many helpful discussions.

\newpage

\section{Donaldson's heat flow on Hermitian manifolds}
\subsection{}
\label{subsec-prep}
Let $(X,g)$ be an Hermitian manifold of complex dimension $n$, and $(E,H)$ be a complex vector bundle over $X$ of rank $r$. When there is no ambiguity, we use the same notation for a vector bundle and the sheaf of its $C^{\infty}$-sections. We use $\Aut(E)$ (resp. $\End(E)$) to denote the automorphism (resp. endomorphism) bundle of $E$, and set
$$
\Aut(E,H):=\{w\in\Aut(E)\mid w^*w=\bb I\},\quad\End(E,H):=\{w\in\End(E)\mid w^*+w=0\}
$$
where $*$ denotes the adjoint taken with respect to metric $H$. The global sections of $\Aut(E)$ (resp. $\Aut(E,H)$) form a group known as the \emph{complex} (resp. \emph{real}) \emph{gauge group}, and we refer to a given section as a \emph{complex} (resp. \emph{real}) \emph{gauge transformation}. Let $\mathcal{A}(E,H)$ be the space of unitary connections on $E$ with respect to $H$. Any section $w\in\Aut(E)$ acts on $\mathcal{A}(E,H)$ by first decomposing $D\in\mathcal{A}(E,H)$ into $D=D'+D''$, its $(1,0)$ and $(0,1)$-parts, and then setting
\begin{equation}
\label{eq-complex-gauge-action}
w(D)=(w^*)^{-1}\circ D'\circ w^*+w\circ D''\circ w^{-1}
\end{equation}
When $w$ is a real gauge transformation, we have the familiar action $w(D)=w\circ D\circ w^{-1}$. A connection $D\in\mathcal{A}(E,H)$ is integrable, if it defines a holomorphic structure on $E$, or equivalently, if $(D'')^2=0$, or if its curvature form $F_D$ is of type $(1,1)$ (cf. \cite[Thm. 2.1.53]{Don-Kro}). It is clear from \eqref{eq-complex-gauge-action} that integrability is preserved under the action of $\Aut(E)$. The group $\Aut(E)$ also acts on the space of Hermitian metrics on $E$ by $w(H)(s_1,s_2)=H(w(s_1),w(s_2))$.

An equivalent setting is provided by the theory of principal bundles. Let $\fr X$ be the underlying real manifold of $X$, with the induced Riemannian metric. Let $P:=U(E,H)$ be the unitary frame bundle of $E$. Then $P$ is a principal bundle over $\fr X$, with real structure group $U(r)$. The space $\mathcal{A}(E,H)$ is canonically isomorphic to the space $\mathcal{A}(P)$ of connections on $P$. Furthermore, the $\mathrm{Ad}$-invariant inner-product $\langle A,B\rangle=\mathrm{tr}(A^*B)$ on $\fr u(r)$ induces a metric on $\End(E,H)$ that is equivalent to the metric given by $H$. In what follows, we will switch freely between the two viewpoints for convenience of the situation.

\medskip

\subsection{} Donaldson's heat flow, introduced in \cite{Don}, is a nonlinear parabolic equation of endomorphisms. More precisely, let us fix an Hermitian holomorphic bundle $(E,\bar{\partial},H_0)$ of rank $r$ over a compact Hermitian manifold $(X,g)$. Then Donaldson \cite{Don} shows that there exist sections $w_t$ of $\Aut(E)$, smoothly parametrized by $t\ge 0$, such that
\begin{enumerate}[(i)]
	\item $w_0=\bb I$, and
	\item if we write $H_t=w_t(H_0)$ and $h_t=w_t^*w_t$ where $*$ denotes the $H_0$-adjoint, there holds
\begin{equation}
\label{eq-donaldson-heat-flow}
h_t^{-1}\frac{d}{dt}h_t=-(i\Lambda F_{(\bar{\partial},H_t)}-\mu(E)\bb I)
\end{equation}
\end{enumerate}
where $F_{(\bar{\partial},H_t)}$ is the curvature form of $D_{(\bar{\partial},H_t)}$, the unique $H_t$-unitary connection whose $(0,1)$-part is $\bar{\partial}$. In this equation, $\mu(E)$ is the Hermitian-Einstein constant of $E$.

Let $D_0=D_{(\bar{\partial},H_0)}$, and it is straightforward to show that $D'_{(\bar{\partial},H_t)}=h_t^{-1}\circ D_0'\circ h_t$. This describes the connection one obtains by fixing a holomorphic structure and vary the metric. On the other hand, fixing the metric $H_0$, we can vary the holomorphic structure on $E$ by applying $w_t$ to $D_0$. The curvature form $F_{w_t(D_0)}$ is related to $F_{(\bar{\partial},H_t)}$ by
\begin{equation}
\label{eq-curvature-relation}
F_{w_t(D_0)}=w_t(D_0)'\circ w_t(D_0)''+w_t(D_0)''\circ w_t(D_0)'=w_tF_{(\bar{\partial},H_t)}w_t^{-1}
\end{equation}
using the integrability of $w_t(D_0)$. The principal motivation of studying the heat flow \eqref{eq-donaldson-heat-flow} in \cite{Don} is that $w_t(D_0)$ is equivalent, modulo real gauge transformations, to the Yang-Mills flow on K\"{a}hler manifolds. When $X$ is only an Hermitian manifold, this equivalence fails, but the same proof gives

\begin{prop}
The smooth family of connections $w_t(D_0)$ is equivalent, modulo real gauge transformations, to a smooth family $D_t\in\mathcal{A}(E,H_0)$ satisfying the equation
\begin{equation}
\tag{\ref{eq-uym}}
\frac{d}{dt}A_t=\frac{i}{2}\left(D''_t\Lambda F_{D_t}-D'_t\Lambda F_{D_t}\right),\quad\forall t\ge 0
\end{equation}
where $D_t=d+A_t$ is the local form of $D_t$. Furthermore, for any family $D_t$ satisfying \eqref{eq-uym}, $D_t$ is unitary and integrable for all $t\ge 0$.
\end{prop}
\begin{proof}
Since this is the same proof as in \cite[\S1]{Don}, we only sketch the computation. Let $\tilde{D}_t:=w_t(D_0)$, and $\tilde{D}_t=d+\tilde{A}_t$ in local form. Using \eqref{eq-donaldson-heat-flow}, \eqref{eq-curvature-relation}, and the identities
$$
\frac{d}{dt}\tilde{A}'_t=\tilde{D}'_t\left((w^*_t)^{-1}\frac{dw^*_t}{dt}\right),\quad\frac{d}{dt}\tilde{A}''_t=-\tilde{D}''_t\left(\frac{dw}{dt}w^{-1}\right)
$$
we may obtain
\begin{equation}
\label{eq-dhf-connection-version}
\frac{d}{dt}\tilde{A}_t=\frac{i}{2}\left(\tilde{D}_t''\Lambda F_{\tilde{D}_t}-\tilde{D}_t'\Lambda F_{\tilde{D}_t}\right)+\tilde{D}_t(\alpha_t)
\end{equation}
where
$$
\alpha_t=\frac{1}{2}\left((w_t^*)^{-1}\frac{dw_t^*}{dt}-\frac{dw_t}{dt}w_t^{-1}\right)\in\End(E,H_0)
$$
We use a real gauge transformation to eliminate the $\tilde{D}_t(\alpha_t)$ term. Define $\theta_t$ by the pointwise exponential map:
$$
\theta_t:=\exp\left(\int_0^t\alpha_sds\right)\in\Aut(E,h)
$$
and set $D_t=\theta_t^*(\tilde{D}_t)$, with $D_t=d+A_t$. It follows from \eqref{eq-dhf-connection-version} that
$$
\frac{d}{dt}A_t=\frac{i}{2}\left(D''_t\Lambda F_{D_t}-D'_t\Lambda F_{D_t}\right)+\theta_t\tilde{D}_t(\alpha_t)\theta_t^*-D_t\left(\frac{d\theta_t}{dt}\theta_t^*\right)
$$
where
$$
\theta_t\tilde{D}_t(\alpha_t)\theta_t^*=D_t(\theta_t\alpha_t\theta_t^*)=D_t\left(\frac{d\theta_t}{dt}\theta_t^*\right)
$$
using the definition of $D_t$ and $\theta_t$.

To prove that for any family $D_t$ satisfying \eqref{eq-uym}, $D_t$ is unitary and integrable for all $t\ge 0$, it suffices to show that $D_t=v_t(D_0)$ for some $v_t\in\Aut(E)$. In fact, the tangent space $T_D$ at any $D\in\mathcal{A}(E,H_0)$ to the $\Aut(E)$-orbit is given by
$$
T_D=\{D'(f^*)-D''(f),\quad\text{for some $f\in\End(E)$}\}
$$
Since $i\Lambda F_D\in\End(E)$ is self-adjoint, the endomorphism form $i(D''\Lambda F_D-D'\Lambda F_D)$ is an element of $T_D$. Therefore, any $D_t$ satisfying \eqref{eq-uym} lies in the $\Aut(E)$-orbit of $D_0$, and the proof is complete.
\end{proof}

\noindent
One important property of the flow equation \eqref{eq-uym} is that $\|\Lambda F_{D_t}\|_{L^{\infty}}$ is uniformly controlled along the flow, as showed by the following

\begin{lem}
\label{lem-lambda-f-lp}
Let $D_t$ be a solution of \eqref{eq-uym}. Then for all $1\le p\le\infty$, the norm $\|\Lambda F_{D_t}\|_{L^p}$ is non-increasing as a function of $t$. In particular, $\|\Lambda F_{D_t}\|_{L^p}$ is uniformly bounded in $t$, for all $1\le p\le\infty$.
\end{lem}
\begin{proof}
Write $D_t=d+A_t$ locally. Then
$$
\frac{d}{dt}\Lambda F_{D_t}=\Lambda\frac{d}{dt}F_{D_t}=\Lambda D_t\left(\frac{d}{dt}A_t\right)
$$
where in the second equality, we have used the local expression $F_{D_t}=dA_t+A_t\wedge A_t$. Using the flow equation \eqref{eq-uym}, and the fact that $D_t$ is integrable, we obtain
\begin{equation}
\label{eq-flow-lambda-f}
\frac{d}{dt}\Lambda F_{D_t}=\frac{i}{2}\Lambda\left((D_t'D_t''-D_t''D_t')\Lambda F_{D_t}\right)
\end{equation}
Consider now the operator $P:=i\Lambda\bar{\partial}\partial$, where $d=\partial+\bar{\partial}$ is the decomposition into its $(1,0)$ and $(0,1)$-parts. The operator $P$ is a second order elliptic operator. Locally, the coefficients in front of the top-order differentials of $P$ form a symmetric, negative definite real matrix, and thus $P$ satisfies the maximum principle (cf. \cite[\S7.2]{Lub} for details). Let $|\Lambda F_{D_t}|$ denote the pointwise norm of $\Lambda F_{D_t}$. We may now compute using \eqref{eq-flow-lambda-f}:
\begin{align*}
\left(\frac{d}{dt}+P\right)|\Lambda F_{D_t}|^2=&2\langle\frac{d}{dt}\Lambda F_{D_t},\Lambda F_{D_t}\rangle+\frac{i}{2}\Lambda\left((\bar{\partial}\partial-\partial\bar{\partial})|\Lambda F_{D_t}|^2\right)\\
=&-2i\Lambda\langle D'_t\Lambda F_{D_t},D''_t\Lambda F_{D_t}\rangle\\
=&-2|D_t'\Lambda F_{D_t}|^2
\end{align*}
which is non-negative. The lemma follows by an application of the maximum principle.
\end{proof}

\medskip

\subsection{}
\label{sec-gauduchon-surface} For any integrable, unitary connection $D$ on the Hermitian vector bundle $(E,H_0)$, Demailly \cite{Dem} proved analogues of the K\"{a}hler identities:
\begin{equation}
\label{eq-dem-commutation-relation}
[\Lambda,D'']=-i((D')^*+\tau^*),\quad[\Lambda,D']=i((D'')^*+\bar{\tau}^*)
\end{equation}
where $\tau=[\Lambda,\partial\omega\wedge]$ is the \emph{torsion operator} of bidegree $(1,0)$. Applying \eqref{eq-dem-commutation-relation} and the Bianchi identities, we see that the curvature form $F_D$ satisfies
\begin{equation}
\label{eq-curvature-commutation-relation}
D^*F_D=iD'\Lambda F_D-iD''\Lambda F_D-(\tau+\bar{\tau})^*F_D
\end{equation}
The extra torsion term in \eqref{eq-curvature-commutation-relation} is responsible for many differences between the limiting behavior of the flow \eqref{eq-uym} and that of the Yang-Mills flow. For instance, one important property of the Yang-Mills flow is that the full curvature form has decreasing $L^2$-norm along the flow. In general, this property cannot be said for the flow \eqref{eq-uym}. However, in certain special cases a good control of $\|F_{D_t}\|_{L^2}$ is still available along the flow \eqref{eq-uym}, as can be seen from the following results.

\begin{lem}
\label{lem-surface-torsion}
Given any $2$-form $\xi$ on an Hermitian surface $X$. There holds
\begin{equation}
\label{eq-surface-torsion}
\Lambda \xi(d^*\omega)=-(\tau+\bar{\tau})^*\xi
\end{equation}
\end{lem}
\begin{proof}
Note that $\tau+\bar{\tau}=[\Lambda,d\omega\wedge]$. Hence for any $1$-form $\eta$,
$$
\langle\eta,(\tau+\bar{\tau})^*\xi\rangle=\langle\Lambda(d\omega\wedge\eta),\xi\rangle=d\omega\wedge\eta\wedge *L\bar\xi=\langle\eta,*(d\omega\wedge *L\xi)\rangle
$$
It follows that
\begin{equation}
\label{eq-torsion-adjoint}
(\tau+\bar\tau)^*\xi=*(d\omega\wedge*L\xi)
\end{equation}
Consider the Lefschetz decomposition $\xi=\xi_2+L\xi_0$, where $\xi_0\in\Lambda^0X$, and $\xi_2\in\Lambda^2X$. Since $X$ is a surface, $\xi_2$ satisfies both $L\xi_2=0$ and $\Lambda\xi_2=0$. Furthermore, using $*(\omega^2)=2$, $\Lambda\xi=2\xi_0$, and $*\omega=\omega$, we compute
$$
(\tau+\bar\tau)^*\xi=*(d\omega\wedge *L^2\xi_0)=*(d\omega\wedge *\xi_0(\omega^2))=2\xi_0(*d\omega)=\Lambda\xi(*d*\omega)=-\Lambda\xi(d^*\omega)
$$
as required.
\end{proof}

\noindent
Since $\Lambda$ and $\tau$ are both linear algebraic operators on forms, \eqref{eq-surface-torsion} applies to the curvature form $F_D$, and we obtain the following form of \eqref{eq-curvature-commutation-relation}:
\begin{equation}
\label{eq-curvature-commutation-relation-surface}
D^*F_D=iD'\Lambda F_D-iD''\Lambda F_D+\Lambda F_D(d^*\omega)
\end{equation}
over any Hermitian surface $X$.

\begin{rem}
\label{rem-surface-torsion}
The merit of Lem. \ref{lem-surface-torsion} is that the extra torsion term in \eqref{eq-curvature-commutation-relation-surface} only involves $\Lambda F_D$ instead of the full curvature form $F_D$. In the appendix, we will say more about the expression $(\tau+\bar{\tau})^*\xi$ in higher dimensions.
\end{rem}

Given an arbitrary Hermitian manifold $(X,g)$ of dimension $n$, in the conformal class of $g$ one can find an Hermitian metric whose associated $(1,1)$-form $\omega$ satisfies $\partial\bar{\partial}(\omega^{n-1})=0$ (cf. \cite{Gau}). Such a metric is called a \emph{Gauduchon metric}, and $(X,g)$ (or simply $X$ when the metric is understood from the context) is called a \emph{Gauduchon manifold}. Equivalently, $X$ is Gauduchon if and only if $(\bar{\partial}\partial)^*\omega=0$, because of the identities 
$$
*\omega=\frac{\omega^{n-1}}{(n-1)!},\quad\text{and}\quad(\bar{\partial}\partial)^*=\partial^*\bar{\partial}^*=\pm *\bar{\partial}\partial*
$$
If $X$ is compact, this latter condition can be rewritten as
\begin{equation}
\label{eq-gauduchon-perpendicularity}
\int_X\langle\partial^*\omega,\bar{\partial}f\rangle=0,\quad\forall f\in\Lambda^0X
\end{equation}
In other words, $\partial^*\omega$ is perpendicular to the $\bar{\partial}$-exact forms in the space $\Lambda^{0,1}X$. Of course, a similar property holds for the $(1,0)$-form $\bar{\partial}^*\omega$. The following proposition is an important consequence of this geometric input.

\begin{prop}
\label{prop-bound-f}
Assume the base manifold $(X,g)$ is a Gauduchon surface, and let $D_t$ be a solution of \eqref{eq-uym}. Then
\begin{equation}
\label{eq-flow-f}
\frac{d}{dt}\|F_{D_t}\|_{L^2}^2=-\|D_t\Lambda F_{D_t}\|_{L^2}^2,\quad\forall t\ge 0
\end{equation}
In particular, $\|F_{D_t}\|_{L^2}$ is non-increasing as a function of $t$.
\end{prop}
\begin{proof}
We first compute the time derivative of the pointwise norm
$$
\frac{d}{dt}|F_{D_t}|^2=2\langle\frac{d}{dt}F_{D_t},F_{D_t}\rangle=2\langle D_t\left(\frac{d}{dt}A_t\right),F_{D_t}\rangle=2\langle\frac{d}{dt}A_t,D_t^*F_{D_t}\rangle
$$
where for the last equality, we have used the local expression $F_{D_t}=dA_t+A_t\wedge A_t$. Using \eqref{eq-curvature-commutation-relation-surface}, we expand the term $D_t^*F_{D_t}$ to get
\begin{align*}
\frac{d}{dt}|F_{D_t}|^2=&2\langle\frac{d}{dt}A_t,iD_t'\Lambda F_{D_t}-iD_t''\Lambda F_{D_t}+\Lambda F_{D_t}(\partial^*\omega)+\Lambda F_{D_t}(\bar{\partial}^*\omega)\rangle\\
=&\langle iD_t''\Lambda F_{D_t}-iD_t'\Lambda F_{D_t},iD_t'\Lambda F_{D_t}-iD_t''\Lambda F_{D_t}+\Lambda F_{D_t}(\partial^*\omega)+\Lambda F_{D_t}(\bar{\partial}^*\omega)\rangle
\\
=&-|D_t\Lambda F_{D_t}|^2-\langle iD_t'\Lambda F_{D_t},\Lambda F_{D_t}(\bar{\partial}^*\omega)\rangle+\langle iD_t''\Lambda F_{D_t},\Lambda F_{D_t}(\partial^*\omega)\rangle
\end{align*}
Note that
$$
\langle iD_t''\Lambda F_{D_t},\Lambda F_{D_t}(\partial^*\omega)\rangle=\frac{i}{2}\langle\bar{\partial}|\Lambda F_{D_t}|^2,\partial^*\omega\rangle,\quad\text{and}\quad\langle iD_t'\Lambda F_{D_t},\Lambda F_{D_t}(\bar{\partial}^*\omega)\rangle=\frac{i}{2}\langle\partial|\Lambda F_{D_t}|^2,\bar{\partial}^*\omega\rangle
$$
both integrate to zero, by \eqref{eq-gauduchon-perpendicularity} and its analogue for $\bar{\partial}^*\omega$. The proposition is proved.
\end{proof}

\begin{cor}
\label{cor-time-sequence}
Assume the base manifold $(X,g)$ is a Gauduchon surface, and let $D_t$ be a solution of \eqref{eq-uym}. Then there is a sequence $t_j\rightarrow\infty$ such that $\|D_{t_j}\Lambda F_{D_{t_j}}\|_{L^2}\rightarrow 0$.
\end{cor}

\noindent
Given a solution $D_t$ of the flow \eqref{eq-uym}, we will call any such sequence $t_j\rightarrow\infty$ a \emph{minimizing sequence}.

\begin{proof}
By integrating \eqref{eq-flow-f} from $t=0$ to $t=T\ge 0$, we obtain
$$
\|F_{D_T}\|_{L^2}^2-\|F_{D_0}\|_{L^2}^2=-\int_0^T\|D_t\Lambda F_{D_t}\|_{L^2}^2dt
$$
Since $\|F_{D_T}\|_{L^2}^2$ is non-increasing as a function of $T$, the right-hand-side is finite as $T\rightarrow\infty$. Hence there exists a sequence $t_j\rightarrow\infty$ such that $\|D_{t_j}\Lambda F_{D_{t_j}}\|_{L^2}\rightarrow 0$.
\end{proof}

\medskip

\section{Critical Hermitian structures and Uhlenbeck limits}
\subsection{} For a fixed holomorphic vector bundle $(E,\bar{\partial})$, a \emph{critical Hermitian structure} on $E$ is a critical point of $\|\Lambda F_{(\bar{\partial},H)}\|_{L^2}^2$, as a functional defined on the space of Hermitian metrics on $E$. By \cite[\S IV, Thm. 3.21]{Kob}, the metric $H$ is a critical Hermitian structure if and only if $D_{(\bar{\partial},H)}\Lambda F_{(\bar{\partial},H)}=0$. It follows from this equation that $E$ admits a direct sum decomposition into holomorphic subbundles $Q^{(i)}$, where each induced metric $H^{(i)}$ is Hermitian-Einstein, i.e. it satisfies $i\Lambda F_{(\bar{\partial},H^{(i)})}=\mu(Q^{(i)})\bb I$ (cf. \cite[\S IV, Thm. 3.27]{Kob}). For a general holomorphic vector bundle, critical Hermitian structures may not exist. Hence along some minimizing sequence $t_j\rightarrow\infty$ (cf. Cor. \ref{cor-time-sequence}), a ``limit" has to be found on a possibly different bundle. The precise meaning of such limits will be given by the concept of Uhlenbeck limits. We will prove in this section that any minimizing sequence $t_j\rightarrow\infty$ has an Uhlenbeck limit which is a critical Hermitian structure. In doing so, several analytic results about the equation
\begin{equation}
\label{eq-chs}
D\Lambda F_D=0
\end{equation}
are required, such as regularity and removability of singularity (on surfaces). In the K\"{a}hler setting, they are available from the analysis of Yang-Mills equation. However, on a general Hermitian manifold, \eqref{eq-chs} differs from the Yang-Mills equation by a torsion term. In the case of Hermitian-Einstein equations, these results are perhaps known to experts. Nevertheless, the authors supply proofs for them.

\medskip

\subsection{}
\label{sec-limit-definition}Let $(E,H)$ be an Hermitian vector bundle of rank $r$ over an Hermitian manifold $(X,g)$. Fix a smooth unitary connection $D_0$ on $E$. Then any (possibly non-smooth) connection $D$ can be written as $D_0+\alpha$ for some endomorphism-valued 1-form $\alpha$. When $X$ is compact, we may define the $W^{k,p}$-norm of $D$ by that of $\alpha$. The space $\mathcal{A}^{k,p}(E,H)$ of unitary connections with bounded $W^{k,p}$-norm is a Banach space. When $X$ is non-compact, we set
$$
\mathcal{A}_{\loc}^{k,p}(E,H):=
\left.
     \begin{array}{lr}
       \text{space of unitary connections on $X$ that}\\
       \text{restrict to $\mathcal{A}^{k,p}(E|_K,H)$ for all compact $K\subset X$}
     \end{array}
   \right.
$$
When $X$ is compact, we also use $\Aut^{k,p}(E,H)$ to denote the Banach space of real gauge transformations with bounded $W^{k,p}$-norm. When $X$ is non-compact, set
$$
\Aut_{\loc}^{k,p}(E,H):=
\left.
     \begin{array}{lr}
       \text{space of real gauge transformations on $X$ that}\\
       \text{restrict to $\Aut^{k,p}(E|_K,H)$ for all compact $K\subset X$}
     \end{array}
   \right.
$$
In the equivalent setting of the principal bundle $P=U(E,H)\rightarrow\fr X$, where $\fr X$ is the underlying Riemannian manifold of $X$, these spaces can also be defined using the trace inner-product on $\fr u(r)$ (cf. \S\ref{subsec-prep}). Wehrheim \cite[Appendix A\&B]{Weh} contains an excellent summary of their properties in relation to Uhlenbeck compactness.

\begin{defn}
\label{defn-uhlenbeck-limit}
Let $(E,H)$ be an Hermitian vector bundle over a compact Hermitian surface $(X,g)$. For any $p\ge 1$, a sequence of connections $D_j$ \emph{converges weakly in $W^{1,p}$ along some subsequence $D_{j_k}$ to an Uhlenbeck limit} $D_{\infty}$, if there exist
\begin{enumerate}[(i)]
	\item a finite set $Z^{\an}\subset X$,
	
	\item a $W^{1,p}_{\loc}$-connection $D_{\infty}$ on the Hermitian vector bundle $(E|_{X-Z^{\an}},H)$ over $X-Z^{\an}$, and
	
	\item a sequence of unitary gauge transformations $\tau_{j_k}\in\Aut^{2,p}_{\loc}(E|_{X-Z^{\an}},H)$ such that
	$$
	\tau_{j_k}(D_{j_k})\rightharpoonup D_{\infty}\quad\text{weakly in $W_{\loc}^{1,p}$}
	$$
	as connections on $E|_{X-Z^{\an}}$.
\end{enumerate}
When $p$ is fixed, we will simply say that $D_{\infty}$ is an Uhlenbeck limit of the sequence $D_j$ along $D_{j_k}$.
\end{defn}

\begin{rem}
\label{rem-defn-uhlenbeck-limit}
By passing to a further subsequence of $D_{j_k}$ if necessary, we may replace (ii)(iii) in the definition by
\begin{enumerate}[(i)]
	\item[(ii${}^*$)] a smooth Hermitian vector bundle $(E_{\infty},H_{\infty})$ over $X-Z^{\an}$ equipped with a $W^{1,p}_{\loc}$-connection $D_{\infty}$, and
	
	\item[(iii${}^*$)] for any compact set $K\subset\subset X-Z^{\an}$, a $W^{2,p}$-isometry $\tau^K:(E_{\infty},H_{\infty})|_K\rightarrow(E,H)|_K$ such that for $K\subset K'\subset\subset X-Z^{\an}$, $\tau^K=\tau^{K'}|_K$, and $\tau^K(D_{j_k})\rightharpoonup D_{\infty}$ weakly in $W^{1,p}(K)$.
\end{enumerate}
This is the definition given in \cite{Das-Wen}, and is apparently stronger than our (ii)(iii). We prove that (ii)(iii) imply (ii${}^*$)(iii${}^*$) in the appendix.
\end{rem}

Assume $(X,g)$ is a compact Hermitian surface. When $D_j$ is a sequence of integrable, unitary connections on $(E,H)$ satisfying some uniform bounds on the curvature, the bubbling set $Z^{\an}$ and a convergent subsequence away from $Z^{\an}$ can be found by the following weak compactness result. Its proof in the K\"{a}hler case can be deduced from Uhlenbeck \cite[Thm. 2.2]{Uhl_a}, but in the non-K\"{a}hler case, we need to substitute the key $L^p$-estimate by an argument of Donaldson.

\begin{prop}
\label{prop-weak-compactness}
Let $(E,h)$ be an Hermitian vector bundle over a compact Hermitian surface $(X,g)$. Suppose $D_j$ is a sequence of integrable, unitary connections on $E$ such that $\|\Lambda F_{D_j}\|_{L^{\infty}}$, $\|F_{D_j}\|_{L^2}$ are uniformly bounded. Fix any $p>2$. Then $D_j$ converges weakly in $W^{1,p}$ along some subsequence $D_{j_k}$ to an Uhlenbeck limit $D_{\infty}$.
\end{prop}
\begin{proof}
Since Uhlenbeck gauge exists for the unit ball $B_1\subset\CC^2$ with an Hermitian metric $g$ which is $W^{2,\infty}$-close to the standard metric (cf. \cite[Thm. 6.3]{Weh}), by a scaling argument we can find constants $\varepsilon$, $C>0$, depending only on $p$ and the geometry of $X$ and $E$, such that the following holds: any $x\in X$ admits a neighborhood $U$ on which any connection $D=d+A\in\mathcal{A}^{1,2}(E|_U,H)$ satisfying $\|F_D\|\le\varepsilon$ can be put in the Uhlenbeck gauge, i.e. there exists $u\in\Aut^{2,n}(E|_{\overline{U}},H)$ such that
$$
d^*(u(A))=0,\quad\text{and}\quad\|u(A)\|_{W^{1,2}(U)}\le C\|F_D\|_{L^2(U)}
$$
where $u(A)=uAu^{-1}-(du)u^{-1}$ is the local connection $1$-form associated to $u(D)$. Let $\pi^+:\Lambda^2X\rightarrow\Lambda^2(X)^+$ be the orthogonal projection onto the self-dual part $\Lambda^2(X)^+$ of $\Lambda^2(X)$. Then the operator
$$
d^*\oplus(\pi^+\circ d):\Lambda^1X\rightarrow\Lambda^0X\oplus\Lambda^2X
$$
is an overdetermined elliptic operator, thus satisfying elliptic regularity (cf. \cite[Appendix III]{Don-Kro}). Assuming furthermore that $D$ is integrable, the self-dual part of $F_D$ is proportional to the contraction $\Lambda F_D$. Hence by further shrinking the constant $\varepsilon$ if necessary, the proof of \cite[Cor. 23]{Don} shows the following estimate:
\begin{equation}
\label{eq-donaldson-a-priori-estimate}
\|F_D\|_{L^p(V)}\le C\left(\|F_D\|_{L^2(U)}+\|\Lambda F_D\|_{L^{\infty}(U)}^4\right)
\end{equation}
for some refinement $V$ of $U$, containing $x$. Consider this $\varepsilon$ as fixed. We now use Sedlacek \cite[Prop. 3.3]{Sed} to find
\begin{enumerate}[(i)]
	\item a finite set $Z^{\an}\subset X$,
	
	\item a subsequence of $D_j$ (still denoted by $D_j$), and
	
	\item a cover $\{U_{\alpha}\}$ of $X-Z^{\an}$, such that over each $U_{\alpha}$, $\|F_{D_j}\|_{L^2(U_{\alpha})}\le\varepsilon$ for $j$ sufficiently large along this subsequence.
\end{enumerate}
Consider an increasing exhaustion $X-Z^{\an}=\bigcup_{k\ge 1}X_k$ by compact subsets which are deformation retracts of $X$. Then each $X_k$ is covered by finitely many open subsets $V_{\alpha}$ such that
$$
\|F_{D_j}\|_{L^p(V_{\alpha})}\le C_{\alpha}\left(\|F_{D_j}\|_{L^2(X_k)}+\|\Lambda F_{D_j}\|_{L^{\infty}(X_k)}^4\right)
$$
for $j$ larger than some $j_{\alpha}$. It thus follows that $\|F_{D_j}\|_{L^p(X_k)}$ is uniformly bounded in $j$. The global gauge transformations $\tau_j\in\Aut^{2,p}_{\loc}(E,H)$ are provided by the patching result \cite[Thm. 7.5]{Weh}. Finally, using the Sobolev embedding $W^{2,p}\hookrightarrow L^{2p}$, we see that $F_{D_j}\rightharpoonup F_{D_{\infty}}$ weakly in $L^p_{\loc}$. Thus $D_{\infty}$ is still integrable. The unitarity condition $dh=A_j^t\cdot h+h\cdot\bar{A}_j$ is also preserved under the weak limit.
\end{proof}

\medskip

\subsection{} The limit connection $D_{\infty}=d+A_{\infty}$ produced by Prop. \ref{prop-weak-compactness} is \emph{a priori} only in $\mathcal{A}^{1,p}_{\loc}(E|_{X-Z^{\an}},H)$ with $p>2$. Its curvature form $F_{D_{\infty}}=dA_{\infty}+A_{\infty}\wedge A_{\infty}$ is in $L^p_{\loc}$. Thus we may only speak of the equation $D_{\infty}\Lambda F_{D_{\infty}}=0$ in a weak sense. We will prove, in rather general context, that every such weak solution is gauge equivalent to a smooth one. We use the notation $\Theta^k(E,H)$ to denote the space of skew-self-adjoint $\End(E)$-valued $k$-forms of the bundle $E$. Note that by restricting such $k$-forms to real tangent vectors, we obtain a canonical isomorphism $\Theta^k(E,H)\cong\Lambda^k(\fr X,\End(E,H))$.

\begin{prop}
\label{prop-regularity}
Let $(X,g)$ be an Hermitian manifold of dimension $n\ge 2$, such that there is an increasing exhaustion $X=\bigcup_{k=1}^{\infty}X_k$ by compact subsets which are deformation retracts of $X$. Let $(E,H)$ be an Hermitian vector bundle over $X$. Then for any $p>n$, and any integrable connection $D\in\mathcal{A}^{1,p}_{\loc}(E,H)$ such that
\begin{equation}
\label{eq-weak-chs}
\int_X\langle\Lambda F_D,D^*\beta\rangle=0,\quad\forall\beta\in C_0^{\infty}(X,\Theta^1(E,H))
\end{equation}
there exists some $u\in\Aut_{\loc}^{2,p}(E,H)$ such that $u(D)$ is smooth, and satisfies $u(D)\Lambda F_{u(D)}=0$.
\end{prop}

\noindent
A $W^{1,p}_{\loc}$-connection $D$ satisfying \eqref{eq-weak-chs} is said to \emph{satisfy $D\Lambda F_D=0$ in the weak sense} over $X$. These connections are invariant under $W^{2,p}_{\loc}$-gauge transformations:

\begin{lem}
\label{lem-gauge-invariance}
Assume $p>n$. If $D\in\mathcal{A}^{1,p}_{\loc}(E,H)$ satisfies $D\Lambda F_D=0$ in the weak sense over $X$, so does $u(A)$ for every $u\in\Aut_{\loc}^{2,p}(E,H)$.
\end{lem}
\begin{proof}
Observe first that under the assumption $p>n$, $D$ extends to a continuous map from the space of $W^{2,p}$-forms to that of $L^{p'}$ forms, where $p'$ is the H\"{o}lder conjugate of $p$. Because the metric $g$ is smooth, the same property holds for $D^*=-*D*$. Therefore, a smooth approximation argument shows that
$$
\int_X\langle\Lambda F_D,D^*\gamma\rangle=0,\quad\forall\gamma\in W^{2,p}_0(X,\Theta^1(E,H))
$$
Now, for any $\beta\in C^{\infty}_0(X,\Theta^1(E,H))$, it follows from $u(D)(\beta)=uD(u^{-1}\beta u)u^{-1}$ that $u(D)^*(\beta)=uD^*(u^{-1}\beta u)u^{-1}$. Hence
$$
\int_X\langle\Lambda F_{u(D)},u(D)^*\beta\rangle=\int_X\langle u\Lambda F_Du^{-1},uD^*(u^{-1}\beta u)u^{-1}\rangle=\int_X\langle\Lambda F_D,D^*(u^{-1}\beta u)\rangle=0
$$
since $u^{-1}\beta u\in W_0^{2,p}(X,\Theta^1(E,H))$ (\cite[Lem. A.7]{Weh}).
\end{proof}

\noindent
Prop. \ref{prop-regularity} will follow from iteration of a local interior estimate. Let $U\subset\CC^n$ be a bounded domain equipped with an Hermitian metric $g$. Let $E=U\times\CC^r$ be a trivial bundle over $U$ with Hermitian metric $H$. We fix a reference smooth connection $\tilde{D}=d+\tilde{A}\in\mathcal{A}(E,H)$ and three numbers $k\in\mathbb{N}$, $p,q>1$ satisfying either
\begin{enumerate}[(i)]
	\item $kp>2n$ and $q=p$, or
	
	\item $k=1$, $n<p<2n$ and $q=\frac{2np}{4n-p}$.
\end{enumerate}
These choices of constants are the same as in Wehrheim's account \cite{Weh} for the Yang-Mills equation. They make possible the following Sobolev embeddings and multiplications:
\begin{equation}
\label{eq-sobolev-analysis}
W^{k,p}\otimes W^{k,p}\rightarrow W^{k,q},\quad W^{k,p}\hookrightarrow W^{k,q},\quad W^{k,q}\hookrightarrow W^{k-1,p},\quad W^{k,p}\otimes W^{k-1,p}\rightarrow W^{k-1,q}
\end{equation}
The last one, for instance, can be proved by showing that there exists some $r>1$ such that both $W^{k,p}\hookrightarrow W^{k-1,r}$ and $W^{k-1,r}\otimes W^{k-1,p}\rightarrow W^{k-1,q}$. The key local interior estimate is given by

\begin{lem}
\label{lem-local-estimate}
Let $D=\tilde{D}+\alpha\in\mathcal{A}^{k,p}(E,H)$ be integrable. Suppose that $\tilde{D}^*\alpha=0$, and $D$ satisfies $D\Lambda F_D=0$ in the weak sense over $U$. Then for every open subset $V\subset\subset U$, there exists a constant $C$ such that
\begin{equation}
\label{eq-local-estimate}
\|\alpha\|_{W^{k+1,q}(V)}\le C\left(1+\|\alpha\|_{W^{k,p}(U)}^3\right)
\end{equation}
\end{lem}
\begin{proof}
Let $\Delta=dd^*+d^*d$ be the Hodge Laplacian on smooth forms. Then for any $\beta\in C_0^{\infty}(U,\Theta^1(E,H))$,
\begin{equation}
\label{eq-compute-laplacian}
\int_U\langle\alpha,\Delta\beta\rangle=\int_U\langle d^*\alpha,d^*\beta\rangle+\int_U\langle d\alpha,d\beta\rangle=\int_U\langle\gamma,\beta\rangle
\end{equation}
for some 1-form $\gamma$, which we will compute explicitly. Write $D=d+A$ and $\tilde{D}=d+\tilde{A}$. Note that $d^*\alpha=\tilde{D}^*\alpha+*[\tilde{A}\wedge *\alpha]=*[\tilde{A}\wedge*\alpha]$, where $[\cdot,\cdot]$ denotes the commutator (or equivalently in the principal bundle setting, the Lie bracket on $\fr u(r)$). Thus
$$
\int_U\langle d^*\alpha,d^*\beta\rangle=\int_U\langle d\left(*[\tilde{A}\wedge *\alpha]\right),\beta\rangle
$$
From $F_D=dA+A\wedge A=dA+\frac{1}{2}[A\wedge A]$, it follows that $d\alpha=F_D-F_{\tilde{D}}-[(\tilde{A}+\frac{1}{2}\alpha)\wedge\alpha]$. Together with $d\beta=D\beta-[A\wedge\beta]$, we obtain
\begin{equation}
\label{eq-compute-laplacian-2}
\int_U\langle d\alpha,d\beta\rangle=\int_U\langle F_D,D\beta\rangle-\int_U\langle F_{\tilde{D}}+[(\tilde{A}+\frac{1}{2}\alpha)\wedge\alpha],D\beta\rangle-\int_U\langle d\alpha,[A\wedge\beta]\rangle
\end{equation}
To compute the first term, recall the commutation relations of Demailly \cite{Dem}:
\begin{equation}
[(D'')^*,L]=i(D'+\tau),\quad [(D')^*,L]=-i(D''+\bar{\tau})
\end{equation}
which are adjoints of the relations \eqref{eq-dem-commutation-relation}. They also hold for $W^{1,p}$-connections by a smooth approximation argument. Since $D$ is integrable, $\langle F_D,D\beta\rangle=\langle F_D,D'\beta^{0,1}\rangle+\langle F_D,D''\beta^{1,0}\rangle$, where $\beta=\beta^{1,0}+\beta^{0,1}$ is the decomposition of $\beta$ into its $(1,0)$ and $(0,1)$-components. The first part
\begin{align*}
\int_U\langle F_D,D'\beta^{0,1}\rangle=&\int_U\langle F_D,-i(D'')^*(L\beta^{0,1})\rangle+\int_U\langle F_D,iL(D'')^*(\beta^{0,1})\rangle+\int_U\langle F_D,-\tau(\beta^{0,1})\rangle\\
=&\int_U\langle -\tau^*(F_D),\beta\rangle
\end{align*}
using the hypothesis, and a weak form of Bianchi identity. Note that $\langle\tau^*(F_D),\beta^{0,1}\rangle=\langle\tau^*(F_D),\beta\rangle$ by degree considerations. Using a similar argument to compute the integral of $\langle F_D,D''\beta^{1,0}\rangle$, we conclude
\begin{equation}
\label{eq-compute-ym-term}
\int_U\langle F_D,D\beta\rangle=\int_U\langle-(\tau+\bar{\tau})^*F_D,\beta\rangle=\int_U\langle-(\tau+\bar{\tau})^*(F_{\tilde{D}}+d\alpha+[(\tilde{A}+\frac{1}{2}\alpha)\wedge\alpha],\beta)\rangle
\end{equation}
The computation for the second and third terms in \eqref{eq-compute-laplacian-2} is completely identical to the one in \cite[proof of Prop. 9.5]{Weh}. Altogether, we find the following expression for $\gamma$ in \eqref{eq-compute-laplacian}:
\begin{align}
\label{eq-express-gamma}
\gamma=&d(*[\tilde{A}\wedge *\alpha])-(\tau+\bar{\tau})^*(F_{\tilde{D}}+d\alpha+[(\tilde{A}+\frac{1}{2}\alpha)\wedge\alpha])-\tilde{D}^*(F_{\tilde{D}}+[(\tilde{A}+\frac{1}{2}\alpha)\wedge\alpha])\notag\\
&+*[\alpha\wedge *(F_{\tilde{D}}+[(\tilde{A}+\frac{1}{2}\alpha)\wedge\alpha])]-*[\tilde{A}^*\wedge *d\alpha]-*[\alpha^*\wedge *d\alpha]
\end{align}
where $\tilde{A}^*$ (resp. $\alpha^*$) denotes the adjoint of the endomorphism part of $\tilde{A}$ (resp. $\alpha$), coupled with the conjugate of its form part. As in \cite{Weh}, we estimate the $W^{k-1,q}$-norm of $\gamma$, using smoothness of $\tilde{A}$ and $F_{\tilde{D}}$:
$$
\|\gamma\|_{W^{k-1,q}}\le C\left(1+\|\alpha\|_{W^{k,q}}+\|[\alpha\wedge\alpha]\|_{W^{k,q}}+\|\alpha\wedge *[\alpha\wedge\alpha]\|_{W^{k-1,q}}+\|\alpha^*\wedge *d\alpha\|_{W^{k-1,q}}\right)
$$
where each individual terms may be estimated using \eqref{eq-sobolev-analysis}:
\begin{align*}
\|\alpha\|_{W^{k,q}}\le& C\|\alpha\|_{W^{k,p}},\quad\|[\alpha\wedge\alpha]\|_{W^{k,q}}\le C\|\alpha\|_{W^{k,p}}^2\\
\|\alpha\wedge*[\alpha\wedge\alpha]\|_{W^{k-1,q}}\le& C\|\alpha\|_{W^{k,p}}\|[\alpha\wedge\alpha]\|_{W^{k-1,p}}\le C\|\alpha\|_{W^{k,p}}\|[\alpha\wedge\alpha]\|_{W^{k,q}}\le C\|\alpha\|_{W^{k,p}}^3\\
\|\alpha^*\wedge *d\alpha\|_{W^{k-1,q}}\le& C\|\alpha\|_{W^{k,p}}\|d\alpha\|_{W^{k-1,p}}\le C\|\alpha\|_{W^{k,p}}^2
\end{align*}
Altogether, $\gamma$ is bounded in $W^{k-1,q}$. By \eqref{eq-compute-laplacian}, $\Delta\alpha=\gamma$ in the weak sense, so interior regularity of the Hodge Laplacian $\Delta$ implies that for every $V\subset\subset U$,
$$
\|\alpha\|_{W^{k+1,q}(V)}\le C(1+\|\alpha\|_{W^{k,p}(U)}+\|\alpha\|_{W^{k,p}(U)}^2+\|\alpha\|_{W^{k,p}(U)}^3+\|\alpha\|_{L^q(U)})
$$
Since $k\ge 1$, and $q\le p$, $\|\alpha\|_{L^q(U)}$ can be absorbed into $\|\alpha\|_{W^{k,p}(U)}$. By the elementary inequality $\|\alpha\|_{W^{k,p}(U)}\le\frac{1}{2}(1+\|\alpha\|_{W^{k,p}(U)}^2)$, the lower-degree terms may also be absorbed into $\|\alpha\|_{W^{k,p}(U)}^3$, giving \eqref{eq-local-estimate}.
\end{proof}

\begin{proof}[Proof of Prop. \ref{prop-regularity}]
The proof is the same as \cite[Cor. 9.6(ii) \& Thm. 9.4]{Weh}, and does not use the equation except for the local estimate \eqref{eq-local-estimate}. We briefly outline the argument. Let $X''\subset X'\subset X$ be compact subsets of $X$, such that $X''$ is contained in the interior of $X'$. Fix a reference smooth connection $\tilde{D}\in\mathcal{A}(E,H)$. If $D=\tilde{D}+\alpha\in\mathcal{A}^{k,p}_{\loc}(E,H)$ with $\tilde{D}^*\alpha=0$. Then Lem. \ref{lem-local-estimate} implies that
$$
\|\alpha\|_{W^{k+1,q}(X'')}\le C\left(1+\|\alpha\|^3_{W^{k,p}(X')}\right)
$$
using a finite open cover of $X'$ which has a refinement that covers $X''$. Indeed, given the increasing exhaustion $X=\bigcup_{k=1}^{\infty}X_k$, a reference smooth connection $\tilde D_k$ can be found on each $X_{k+1}$ with the property that
$$
\tilde{D}^*(u_k(D)-\tilde{D})=0
$$
for some $u_k\in\Aut^{2,p}(E|_{X_{k+1}},H)$ (cf. \cite[Thm. 8.1 \& Lem. 8.4(iii)]{Weh}). By Lem. \ref{lem-gauge-invariance}, $u_k(D)$ still satisfies $u_k(D)\Lambda F_{u_k(D)}=0$ in the weak sense over the interior of $X_{k+1}$. Choose a sequence $\{X_k^{(l)}\}$ of compact Hermitian manifolds with smooth boundary, such that
$$
X_k\subset\cdots\subset X_k^{l+1}\subset X_k^l\subset\cdots\subset X_k^1=X_{k+1}
$$
and each $X_k^{l+1}$ is in the interior of $X_k^l$. Then an induction shows that $u_k(A)|_{X_k^l}$ is in $W^{l,p}$ for all $l\ge 1$. The global gauge transformation $u\in\Aut_{\loc}^{2,p}(E,H)$ can be constructed using the patching result \cite[Prop. 9.8]{Weh}. By Lem. \ref{lem-gauge-invariance} again, $u(D)$ satisfies
$$
\int_X\langle u(D)\Lambda F_{u(D)},\beta\rangle=\int_X\langle\Lambda F_{u(D)},u(D)^*(\beta)\rangle=0,\quad\forall\beta\in C_0^{\infty}(X,\Theta^1(E,H))
$$
Hence $u(D)\Lambda F_{u(D)}=0$ as desired.
\end{proof}

\begin{rem}
\label{rem-local-regularity}
The above argument by compact exhaustion also shows that in the local setting of Lem. \ref{lem-local-estimate}, the $1$-form $\alpha$ is smooth on every open subset $V\subset\subset U$. This result applies, in particular, when the reference connection $\tilde{D}$ is trivial. We thus see that a connection $D=d+A$ of a trivial bundle over $U\subset\CC^n$ is smooth on $V$, provided that $D\Lambda F_D=0$ in the weak sense over $U$, and $D$ is in the Coulomb gauge, i.e. $d^*A=0$.
\end{rem}

\medskip

\subsection{} Let $(X,g)$ be a compact Hermitian surface. Given a sequence $D_j$ of integrable, unitary connections as in Prop. \ref{prop-weak-compactness}, we obtain a limiting $W^{1,p}_{\loc}$-connection $D_{\infty}$ over $X-Z^{\an}$. Since $Z^{\an}$ is a finite set of points, $X-Z^{\an}$ can be exhausted by an increasing sequence of compact subsets. If, furthermore, the limiting connection $D_{\infty}$ satisfies $D_{\infty}\Lambda F_{D_{\infty}}$ in the weak sense, then Prop. \ref{prop-regularity} applies and gives a gauge transformation $u\in\Aut_{\loc}^{2,p}(E|_{X-Z^{\an}},H)$ such that $\tilde{D}_{\infty}:=u(D_{\infty})$ is smooth, and satisfies $\tilde{D}_{\infty}\Lambda F_{\tilde{D}_{\infty}}=0$ in the strong sense. To obtain a limiting Hermitian vector bundle over $X$, we need to extend the connection $\tilde{D}_{\infty}$ across $Z^{\an}$. This removability of singularity is a purely local property of our equation. Denote by $B(x,r)\subset\CC^2$ the closed ball $B(x,r)=\{y\in\CC^n:|x-y|\le r\}$.

\begin{lem}
\label{lem-removable-singularity}
Let $B(0,2)\subset\CC^2$ be equipped with an Hermitian metric $g$. Let $(E,H)$ be an Hermitian vector bundle over $B(0,2)-\{0\}$. There exists a constant $\varepsilon>0$, such that the following holds:

Let $D$ be any smooth connection on $E$ with $D\Lambda F_D=0$, and $\|F_D\|_{L^2(B(0,2))}\le\varepsilon$. Then there exists a gauge in which $(E,H)$ extends to a smooth bundle $(\overline{E},\overline{H})$ over $B(0,2)$, and $D$ extends to a smooth connection $\overline{D}$ on $\overline{E}$ with $\overline{D}\Lambda F_{\overline{D}}=0$.
\end{lem}

\noindent
Note that in the geometric setting where $\|F_D\|_{L^2(X)}$ is bounded, by shrinking the neighborhood $U$ around each point singularity $z\in X$, we may always assume that $\|F_D\|_{L^2(U)}$ is sufficiently small. By pulling back to $B(0,2)\subset\CC^2$, Lem. \ref{lem-removable-singularity} immediately implies
\begin{prop}
\label{prop-removable-singularity}
Let $(X,g)$ be a compact Hermitian surface, and $Z\subset X$ be a finite set of points. Let $(E,H)$ be an Hermitian vector bundle over $X-Z$, equipped with a connection $D$ satisfying $D\Lambda F_D=0$, and $\|F_D\|_{L^2(X-Z)}$ is bounded. Then there exists a vector bundle $(\overline{E},\overline{H})$ equipped with a connection $\overline{D}$ with $\overline{D}\Lambda F_{\overline{D}}=0$, such that $(\overline{E}|_{X-Z},\overline{H})\cong(E|_{X-Z},H)$, and under this identification, $\overline{D}$ agrees with $D$.\qed
\end{prop}

The removability of singularity for Yang-Mills connections is proved by Uhlenbeck in \cite{Uhl_r}. In the similar local setting over $B(0,2)\subset\RR^4$, she uses the Yang-Mills equation to obtain that $\|F_D\|_{L^{\infty}(B(0,2))}$ is finite. Then the extension of $E$ and $D$ follows from an application of the implicit function theorem, together with the regularity result that a Yang-Mills connection in the Coulomb gauge is smooth (cf. \cite[proof of Thm. 4.9]{Uhl_r}). For the equation $D\Lambda F_D=0$, this regularity result is precisely Rmk. \ref{rem-local-regularity}. Therefore we only need to prove that $\|F_D\|_{L^{\infty}(B(0,2))}$ is finite, provided that $D\Lambda F_D=0$.

\begin{lem}
\label{lem-local-f-bound}
Let $B(x_0,2a_0)\subset\CC^2$ be equipped with an Hermitian metric $g$, and $(E,H)$ be an Hermitian vector bundle over $B(x_0,2a_0)$. Then there exists a constant $\varepsilon'$ such that if $D$ is any connection on $E$ with $D\Lambda F_D=0$, and $\|F_D\|_{L^2(B(x_0,a_0))}\le\varepsilon'$, then for all $B(x,a)\subset B(x_0,a_0)$,
\begin{equation}
\label{eq-bound-curvature}
|F_D(x)|^2\le Ca^{-4}\int_{B(x,a)}|F_D|^2
\end{equation}
for some positive constant $C$ depending continuously on $g$.
\end{lem}
\begin{proof}
We show that
\begin{equation}
\label{eq-laplacian-f}
\Delta|F_A|\ge -C_1|F_A|^2-C_2|F_A|
\end{equation}
for some positive constants $C_1$, $C_2$ depending continuously on $g$. By letting $f=|F_A|$ and $b=C_1|F_A|+C_2$, the lemma will follow from the argument in \cite[Thm. 3.5]{Uhl_r}. Let $\Box=D^*D+DD^*$ be the exterior derivative Laplacian, and $\Delta=\nabla^*\nabla$ be the Laplacian associated to the full covariant derivative $\nabla$. The computation in \cite[Lem. 3.1]{Uhl_r} shows that
$$
|F_D|\Delta|F_D|\ge\langle F_D,\{F_D,F_D\}\rangle+\langle F_D,\{R,F_D\}\rangle+\langle F_D,\Box F_D\rangle
$$
where $R$ is the curvature $2$-form of the base manifold $B(x_0,2a_0)$, and the notation $\{\cdot,\cdot\}$ denotes a certain bilinear combination with bounded coefficients. Using Bianchi identity, the equation $D\Lambda F_D=0$, and \eqref{eq-curvature-commutation-relation-surface}, we obtain
$$
|\Box F_D|=|DD^*F_D|=|D\left(\Lambda F_D(d^*\omega)\right)|=|\Lambda F_D D(d^*\omega)|\le C_3|F_D|
$$
where $C_3$ is a constant depending continuously on $g$. Thus $|F_D|\Delta|F_D|\ge -C_1|F_D|^3-C_2|F_D|^2$. \eqref{eq-laplacian-f} follows by dividing $|F_D|$.
\end{proof}

Consider now an Hermitian vector bundle $(E,H)$ defined over $B(0,2)-\{0\}$. As in \cite[Lem. 4.5]{Uhl_r}, the bound \eqref{eq-bound-curvature} on the curvature form implies that if $D\Lambda F_D=0$ and $\|F_D\|_{L^2(B(0,2))}\le\varepsilon'$, then
\begin{equation}
\label{eq-bound-curvature-singularity}
|F_D(x)|^2\le C|x|^{-4}\int_{B(0,2|x|)}|F_D|^2
\end{equation}
for all $x\in B(0,2)$ with $0<|x|\le 1$. By the discussion before Lem. \ref{lem-local-f-bound}, the following result will complete the proof of Lem. \ref{lem-removable-singularity}.

\begin{lem}
There exists a constant $\varepsilon>0$ such that if $D$ is any connection on $E$ with $D\Lambda F_D=0$, and $\|F_D\|_{L^2(B(0,2))}\le\varepsilon$, then $\|F_D\|_{L^{\infty}(B(0,2))}$ is finite.
\end{lem}
\begin{proof}
Write
$$
f(r):=\int_{B(0,r)}|F_D|^2,\quad 0\le r\le 1
$$
We first prove the differential inequality
\begin{equation}
\label{eq-differential-inequality}
(1-C_1f(2r)^{1/2}-C_2r)f(r)\le\frac{rf'(r)}{4}
\end{equation}
for all $0<r\le 1$, where $C_1$, $C_2$ are positive constants. This is analogous to \cite[Prop. 4.7]{Uhl_r} except for the extra term due to torsion. To use the gauge theoretical results there, we view $D=d+A$ as a principal bundle connection over $\fr B(0,2)$, the underlying Riemannian ball of $B(0,2)$. As in \cite{Uhl_r}, we apply the broken Hodge gauge and get $A(l)$, $F_D(l)$ defined on the annuli $\fr A_l:=\{x:2^{-l-1}r \le |x|\le 2^{-l}r\}$ for $l=0,1,2,\cdots$. Let $S_l:=\{x:|x|=2^{-l}r\}$ for $l=0,1,2,\cdots$. For $0<r\le 1$, the computation in \cite{Uhl_r} yields
\begin{equation}
\label{eq-differential-inequality-estimate}
f(r)\le (2-C_3f(2r)^{\frac{1}{2}})^{-1}\left(\frac{rf'(r)}{2}\right)+C_4f(2r)^{\frac{1}{2}}f(r)+\sum_{l\ge 1}\bigg|\int_{\fr A_l}\langle A(l),D^*F_D(l)\rangle\bigg|
\end{equation}
for positive constants $C_3$, $C_4$. The third term (vanishing in the Yang-Mills case \cite{Uhl_r}) admits an estimate using \eqref{eq-curvature-commutation-relation-surface} and the equation $D\Lambda F_D=0$:
$$
\bigg|\int_{\fr A_l}\langle A(l),D^*F_D(l)\rangle\bigg|=\bigg|\int_{\fr A_l}\langle A(l),(d^*\omega)\Lambda F_D(l)\rangle\bigg|\le C_5\left(\int_{\fr A_l}|A(l)|^2\right)^{\frac{1}{2}}\left(\int_{\fr A_l}|F_D(l)|^2\right)^{\frac{1}{2}}
$$
We then use \cite[Cor. 2.9]{Uhl_r} to estimate the first factor on the right-hand-side, and then apply \eqref{eq-bound-curvature-singularity} to get
$$
\bigg|\int_{\fr A_l}\langle A(l),D^*F_D(l)\rangle\bigg|\le C_6r\int_{\fr A_l}|F_D(l)|^2
$$
for positive constant $C_6$. Here $r$ comes from scaling the inequality in \cite[Cor. 2.9]{Uhl_r}. Substitute this in \eqref{eq-differential-inequality-estimate}, and we find
$$
f(r)\le (2-C_3f(2r)^{\frac{1}{2}})^{-1}\left(\frac{rf'(r)}{2}\right)+C_4f(2r)^{\frac{1}{2}}f(r)+C_6rf(r)
$$
Rearranging this to obtain
$$
(2-C_3f(2r)^{\frac{1}{2}})\left(1-C_4f(2r)^{\frac{1}{2}}-C_6r\right)f(r)\le\frac{rf'(r)}{2}
$$
Then expand the left-hand-side, and discard positive terms. Rename constants to obtain \eqref{eq-differential-inequality}.

By the choice of $\varepsilon$, we may further assume that $\gamma:=1-C_1\varepsilon>0$. We first use $f(2r)^{\frac{1}{2}}\le\varepsilon$ to obtain from \eqref{eq-differential-inequality}:
$$
(\gamma-C_2r)f(r)\le\frac{rf'(r)}{4},\quad\text{hence}\quad\frac{4\gamma}{r}-4C_2\le\frac{f'(r)}{f(r)}
$$
Integrate from $r$ to $1$, exponentiate and we obtain
$$
f(r)\le r^{4\gamma}e^{4C_2(1-r)}f(1)\le r^{4\gamma}e^{4C_2}\varepsilon^2,\quad\text{for}\quad 0<r\le 1
$$
Use this estimate along with \eqref{eq-differential-inequality},
$$
\frac{4(1-C_1(2r)^{2\gamma}e^{2C_2}\varepsilon)}{r}-4C_2\le\frac{f'(r)}{f(r)}
$$
Again integrate from $r$ to $1$ and exponentiate,
$$
f(r)\le r^4e^{C_1\frac{2^{2C_2}}{2C_2}(1-r^{2C_2})e^{2C_2}C_1+4C_2(1-r)}f(1)\le C_7r^4,\quad\text{for}\quad 0<r\le 1
$$
where $C_7=e^{C_1\frac{2^{2C_2}}{2C_2}e^{2C_2}\varepsilon+4C_2}\varepsilon^2>0$. Finally, by \eqref{eq-bound-curvature-singularity}, we have
$$
|F_D(x)|^2\le K|x|^{-4}\int_{B(0,2|x|)}|F_D|^2=K|x|^{-4}f(2|x|)\le C_8,\quad\text{for}\quad |x|\le\frac{1}{2}
$$
Since $F_D$ is continuous on the compact set $\{x:\frac{1}{2}\le|x|\le2\}$, the norm $\|F_D\|_{L^{\infty}(B(0,2))}$ is finite.
\end{proof}

\medskip

\subsection{}
We summarize the previous analytic results in the following
\begin{prop}
\label{prop-uhlenbeck-limit}
Let $(E,H)$ be an Hermitian vector bundle over a compact Hermitian surface $(X,g)$. Assume $D_j$ is a sequence of integrable, unitary connections on $E$ such that $\|\Lambda F_{D_j}\|_{L^{\infty}}$ and $\|F_{D_j}\|_{L^2}$ are uniformly bounded. Assume furthermore that $\|D_j\Lambda F_{D_j}\|_{L^2}\rightarrow 0$.

Fix some $p>4$. Then $D_j$ converges weakly in $W^{1,p}$ along some subsequence $D_{j_k}$ to an Uhlenbeck limit $D_{\infty}$ (cf. Defn. \ref{defn-uhlenbeck-limit} and Rmk. \ref{rem-defn-uhlenbeck-limit}), on some Hermitian vector bundle $(E_{\infty},H_{\infty})$ over $X$. Furthermore, the triple $(E_{\infty},D_{\infty}'',H_{\infty})$ extends to a smooth holomorphic vector bundle over $X$, and the extended connection $D_{\infty}$ satisfies $D_{\infty}\Lambda F_{D_{\infty}}=0$.
\end{prop}
\begin{proof}
The proof is essentially the same as \cite[Prop. 2.11]{Das-Wen}. The hypothesis $p>4$ gives the Sobolev embedding $W^{1,p}\hookrightarrow C^0$. Hence the subsequence $D_{j_k}$ produced by Prop. \ref{prop-weak-compactness} converges to some $D_{\infty}$ strongly in $C^0_{\loc}$ away from a bubbling set $Z^{\an}$. Together with $\Lambda F_{D_{j_k}}\rightharpoonup\Lambda F_{D_{\infty}}$ weakly in $L^p_{\loc}$, and $\|D_j\Lambda F_{D_j}\|_{L^2}\rightarrow 0$, we can deduce that $D_{\infty}\Lambda F_{D_{j_k}}=D_{j_k}\Lambda F_{D_{j_k}}+[D_{\infty}-D_{j_k},\Lambda F_{D_{j_k}}]\rightarrow 0$ in $L^2_{\loc}$. Thus for all $\beta\in C^{\infty}_0(X,\Theta^1(E,H))$,
$$
\int_X\langle\Lambda F_{D_{\infty}},D_{\infty}^*\beta\rangle=\lim_{k\rightarrow\infty}\int_X\langle \Lambda F_{D_{j_k}},D_{\infty}^*\beta\rangle=\lim_{k\rightarrow\infty}\int_X\langle D_{\infty}\Lambda F_{D_{j_k}},\beta\rangle=0
$$
In other words, $D_{\infty}$ satisfies $D_{\infty}\Lambda F_{D_{\infty}}=0$ in the weak sense. The result then follows from regularity (Prop. \ref{prop-regularity}) and removability of singularity (Prop. \ref{prop-removable-singularity}) for the weak equation $D_{\infty}\Lambda F_{D_{\infty}}=0$.
\end{proof}

\noindent
As in \cite[Cor. 2.12]{Das-Wen}, it also follows that along the subsequence $j_k$,
\begin{equation}
\label{eq-lambda-f-lp-limit}
\Lambda F_{D_{j_k}}\rightarrow\Lambda F_{D_{\infty}},\quad\text{for all}\quad 1\le p<\infty
\end{equation}
\begin{prop}
Let $D_{\infty}$ be some Uhlenbeck limit from Prop. \ref{prop-uhlenbeck-limit}, on Hermitian vector bundle $(E_{\infty},H_{\infty})$. Then $D_{\infty}$ is an integrable, unitary connection, and the holomorphic vector bundle $(E_{\infty},D''_{\infty},H_{\infty})$ has a holomorphic splitting
$$
(E_{\infty},D''_{\infty},H_{\infty})=\bigoplus_i (Q_{\infty}^{(i)},D_{\infty}''^{(i)},H_{\infty}^{(i)})
$$
where $H_{\infty}^{(i)}$ is the Hermitian-Einstein metric on $(Q_{\infty}^{(i)},D_{\infty}''^{(i)})$.
\end{prop}
\begin{proof}
The fact that $D_{\infty}$ is integrable and unitary follows from Prop. \ref{prop-weak-compactness}, and the fact that integrability and unitarity are both preserved under a unitary gauge transformation. The fact that $D_{\infty}\Lambda F_{D_{\infty}}=0$ shows that $H_{\infty}^{(i)}$ is a critical Hermitian structure on $(E_{\infty},D''_{\infty})$, and thus the holomorphic splitting follows from \cite[\S IV, Thm. 3.27]{Kob}.
\end{proof}

\begin{rem}
Given a sequence $D_j$ satisfying the hypothesis of Prop. \ref{prop-uhlenbeck-limit}, its Uhlenbeck limit $D_{\infty}$ is \emph{a prior} non-unique, and the bubbling set $Z^{\an}$ depends on the subsequence $D_{j_k}$.
\end{rem}

\medskip

\subsection{}
To conclude this section, we study the behavior of other Hermitian Yang-Mills type functionals considered by Daskalopoulos and Wentworth \cite{Das-Wen} along the flow \eqref{eq-uym}, as an application of the previous analytic results. 


Let $(E,H)$ be an Hermitian vector bundle of rank $r$, over a compact Gauduchon surface $(X,g)$. Recall that we have the normalization $\vol(X)=2\pi$. For any $\alpha\ge 1$, define a function $\varphi_{\alpha}:\mathfrak{u}(r)\rightarrow\mathbb{R}$ by $\varphi_{\alpha}(\fr a):=\sum_{j=1}^r|\lambda_j|^{\alpha}$ if $\fr a$ has eigenvalues $i\lambda_1,\cdots,i\lambda_r$. Given an integrable, unitary connection $D$ and $N\in\mathbb{R}$, define (following notations of \cite{Das-Wen})
$$
\HYM_{\alpha,N}(D):=\int_X\varphi_{\alpha}(\Lambda F_D+iN\mathbb{I})
$$
and set $\HYM_{\alpha}(D):=\HYM_{\alpha,0}(D)$. Note that the functional $\HYM_2(D)=\|\Lambda F_D\|_{L^2}^2$ and will be abbreviated as $\HYM(D)$. By abuse of notation, for any $r$-tuple $\vec{\mu}$ of real numbers, set $\HYM_{\alpha,N}(\vec{\mu})=2\pi\varphi_{\alpha}(i(\vec{\mu}-N))$, where $\vec{\mu}-N=(\mu_1-N,\cdots,\mu_r-N)$ is identified with the diagonal matrix with entries $\mu_1-N,\cdots,\mu_r-N$. Therefore, if $i\Lambda F_D$ has constant eigenvalues $\mu_1,\cdots,\mu_r$, then $\HYM_{\alpha,N}(D)=\HYM_{\alpha,N}(\vec{\mu})$ for $\vec{\mu}=(\mu_1,\cdots,\mu_r)$.

\begin{prop}
\label{prop-limit-other-hym}
Let $D_t$ be a solution of \eqref{eq-uym} over a compact Gauduchon surface $(X,g)$, and $D_{\infty}$ be an Uhlenbeck limit along any minimizing sequence $t_j\rightarrow\infty$. Then for any $\alpha\ge 1$ and $N\in\mathbb{R}$,
$$
\lim_{t\rightarrow\infty}\HYM_{\alpha,N}(D_t)=\HYM_{\alpha,N}(D_{\infty})
$$
\end{prop}

\noindent
Indeed, such an Uhlenbeck limit exists: by Lem. \ref{lem-lambda-f-lp} and Prop. \ref{prop-bound-f}, $\|F_{D_t}\|_{L^2}$ and $\|\Lambda F_{D_t}\|_{L^{\infty}}$ are uniformly bounded along any solution $D_t$ of \eqref{eq-uym}. Furthermore, along a minimizing sequence (cf. Cor. \ref{cor-time-sequence}), $\|D_{t_j}\Lambda F_{D_{t_j}}\|_{L^2}\rightarrow 0$. Hence Prop. \ref{prop-uhlenbeck-limit} applies and an Uhlenbeck limit $D_{\infty}$ exists on some limiting bundle $(E_{\infty},H_{\infty})$. To prove Prop. \ref{prop-limit-other-hym}, we first make the following generalization of Lem. \ref{lem-lambda-f-lp}. It is proved for the Yang-Mills flow in \cite[Prop. 2.25]{Das-Wen}.

\begin{lem}
\label{lem-other-hym-functional}
Let $D_t$ be a solution of \eqref{eq-uym}. Then for all $\alpha\ge 1$ and $N\in\mathbb{R}$, $\HYM_{\alpha,N}(D_t)$ is non-increasing as a function of $t$.
\end{lem}
\begin{proof}
As in \cite{Das-Wen}, let $\varphi_{\alpha,\rho}:\fr u(r)\rightarrow \RR$ $(0<\rho\le 1)$ be a sequence of smooth convex $\mathrm{ad}$-invariant functions converging uniformly to $\varphi_{\alpha}$, on compact subsets of $\fr u(r)$ as $\rho\rightarrow 0$. As $X$ is compact, it suffices by maximum principle for the operator $P=i\Lambda\bar{\partial}\partial$ (cf. \cite[\S7.2]{Lub}) to show that
\begin{equation}
\label{eq-other-hym-functional}
\left(\frac{d}{dt}+P\right)\varphi_{\alpha,\rho}(\Lambda F_{D_t}+iN\mathbb I)\le 0
\end{equation}
We write $\varphi:=\varphi_{\alpha,N}$ and $f:=\Lambda F_{D_t}+iN\mathbb{I}$ for notational simplicity. Since $P=\frac{i}{2}\Lambda(\bar{\partial}\partial-\partial\bar{\partial})$, we will first compute $\bar{\partial}\partial(\varphi(f(x)))$. Locally, $f$ can be identified as a smooth function $f:U\rightarrow\fr u(r)$ for a domain $U\subset\CC^n$. Using $\fr u(r)\cong\mathbb{R}^{r(r-1)}$, we write $f(x)=(f^1(x),\cdots,f^{r(r-1)}(x))$. For any fixed $A\in\fr u(r)$, the adjoint
\begin{equation}
\label{eq-other-hym-functional-1}
\mathrm{ad}_A(f)^i=[A,f]^i=\sum_{j=1}^{r(r-1)}a^i{}_jf^j
\end{equation}
for some matrix $a=(a^i{}_j)$. Note that since $\varphi$ is $\mathrm{Ad}$-invariant, by picking a path $g(t)\in U(r)$ with $g(0)=\mathbb{I}$ and $g'(0)=A$, we deduce via differentiating $\varphi\circ\mathrm{Ad}_{g(t)}=\varphi$ with respect to $t$:
\begin{equation}
\label{eq-other-hym-functional-2}
\sum_{i=1}^{r(r-1)}(\partial_i\varphi)_{\xi}\cdot a^i{}_j=0,\quad \sum_{i=1}^{r(r-1)}(\partial_k\partial_i\varphi)_{\xi}\cdot a^i{}_j=0,\quad\forall\xi\in\fr u(r)
\end{equation}
for all $1\le j\le r(r-1)$, where the second equation follows from differentiating again with respect to the $k$th coordinate of $\fr u(r)$. Using \eqref{eq-other-hym-functional-2}, for all $1\le l\le 2n$,
$$
\frac{\partial}{\partial x_l}(\varphi(f(x)))=\sum_{i=1}^{r(r-1)}(\partial_i\varphi)_{f(x)}\cdot\left(\frac{\partial f^i}{\partial x_l}\right)_x=\sum_{i,j=1}^{r(r-1)}(\partial_i\varphi)_{f(x)}\cdot\left(\left(\frac{\partial f^i}{\partial x_l}\right)_x+a^i{}_j\cdot (f^j)_x\right)
$$
In particular, if we let $a^i{}_j$ be associated to the adjoint action by the connection form $A(\frac{\partial}{\partial x_l})$ as in \eqref{eq-other-hym-functional-1}, we find
$$
d(\varphi(f(x)))=\sum_{i=1}^{r(r-1)}(\partial_i\varphi)_{f(x)}\cdot\left((Df^i)_x\right)
$$
Taking the $(1,0)$-part and further differentiating by $\bar{\partial}$, we obtain
\begin{align*}
\bar{\partial}\partial(\varphi(f(x)))=&\sum_{i,j=1}^{r(r-1)}(\partial_j\partial_i\varphi)_{f(x)}\cdot(\bar{\partial}f^j)_x\wedge (D'f^i)_x+\sum_{i=1}^{r(r-1)}(\partial_i\varphi)_{f(x)}\cdot(\bar{\partial}D'f^i)_x\\
=&\sum_{i,j=1}^{r(r-1)}(\partial_j\partial_i\varphi)_{f(x)}\cdot(D''f^j)_x\wedge (D'f^j)_x+\sum_{i=1}^{r(r-1)}(\partial_i\varphi)_{f(x)}\cdot (D''D'f^i)_x
\end{align*}
where for the last equality, we have used both formulae in \eqref{eq-other-hym-functional-2}. Now, since $\varphi$ is convex, the matrix $(\partial_l\partial_k\varphi)_{f(x)}$ is positive-definite, and we have
$$
i\Lambda\bar{\partial}\partial(\varphi(f(x)))\le\varphi_{f(x)}'(i\Lambda D''D'f(x)),\quad\text{similarly}\quad -i\Lambda\partial\bar{\partial}(\varphi(f(x)))\le -\varphi_{f(x)}'(i\Lambda D'D''f(x))
$$
Thus
$$
P\left(\varphi(f(x))\right)\le\varphi_{f(x)}'\left(\frac{i}{2}\Lambda(D''D'-D'D'')f(x)\right)=-\varphi_{f(x)}'\left(\frac{d}{dt}f(x)\right)=-\frac{d}{dt}\left(\varphi(f(x))\right)
$$
where we have used \eqref{eq-flow-lambda-f} for the first equality.
\end{proof}

\begin{proof}[Proof of Prop. \ref{prop-limit-other-hym}]
Since $D_{\infty}$ is the Uhlenbeck limit along the subsequence $D_{t_j}$, it follows from \eqref{eq-lambda-f-lp-limit} that $\Lambda F_{D_{t_j}}\rightarrow\Lambda F_{D_{\infty}}$ in $L^{\alpha}$ for all $1\le\alpha<\infty$. By \cite[Lem. 2.23]{Das-Wen}, this convergence is equivalent to $\HYM_{\alpha,N}(D_{t_j})\rightarrow\HYM_{\alpha,N}(D_{\infty})$. Since $\HYM_{\alpha,N}(D_t)$ is non-increasing along the flow, we see that $\HYM_{\alpha,N}(D_t)\rightarrow\HYM_{\alpha,N}(D_{\infty})$.
\end{proof}

\medskip

\section{Identifying the Uhlenbeck limit}
\subsection{} The method of identifying the Uhlenbeck limit along a minimizing sequence of connections $D_j$ (cf. Cor. \ref{cor-time-sequence}) is directly adapted from \cite{Das-Wen}. We perform no more than the act of removing the K\"{a}hler condition from the arguments in \cite{Das-Wen}.

We first give an overview of the Harder-Narasimhan theory of coherent sheaves on Gauduchon manifolds. The main reference for this part is \cite[\S V]{Kob}. Although the results there are stated for K\"{a}hler manifolds, the reason for this constraint lies only in the definition of degree for a coherent sheaf. The theory itself is purely algebraic, and all proofs extend verbatim to Gauduchon manifolds.

Let $(X,g)$ be a compact Gauduchon manifold of dimension $n$ (cf. \S\ref{sec-gauduchon-surface}), and $(E,\bar{\partial})$ be a holomorphic vector bundle over $X$. Define the \emph{degree} of $(E,\bar{\partial})$ as
$$
\deg(E):=\int_Xc_1(E,H)\wedge\omega^{n-1}
$$
where $c_1(E,H)$ is the Chern form associated to any Hermitian metric $H$ on $E$. The degree of $E$ is well-defined independently of $H$, because for any other Hermitian metric $H'$, the difference $c_1(E,H)-c_1(E,H')$ is $\bar{\partial}\partial$-exact, whereas $\omega^{n-1}$ is $\partial\bar{\partial}$-closed. Let $F$ be a coherent sheaf over $\mathcal{O}_X$ of rank $r$, where $\mathcal{O}_X$ is the sheaf of holomorphic functions on $X$. Then its \emph{degree} is defined by
$$
\deg(F)=\deg(\det(F))
$$
where $\det(F)=\left(\bigwedge^rF\right)^{**}$ is the determinant line bundle of $F$ (cf. \cite[\S V.6]{Kob}). This definition is, of course, compatible with the degree of a holomorphic vector bundle, regarded as a coherent sheaf. The \emph{slope} of $F$ is defined as the degree over rank ratio:
$$
\mu(F):=\frac{\deg(F)}{\rank(F)}
$$
Due to the normalization $\vol(X)=2\pi$, $\mu(F)$ agrees with the Hermitian-Einstein constant of $F$. From now on, by a coherent sheaf we will always mean a coherent sheaf over $\mathcal{O}_X$. A torsion-free coherent sheaf $F$ is \emph{semi-stable} if
$$
\mu(S)\le\mu(F),\quad\text{for any coherent subsheaf $S$ of $F$}
$$
and is \emph{stable} if
$$
\mu(S)<\mu(F),\quad\text{for any coherent subsheaf $S$ of $F$ with strictly smaller rank}
$$
In fact, any coherent subsheaf of a torsion-free coherent sheaf is necessarily torsion-free. In checking the (semi-)stability of a coherent sheaf, we may restrict our attention to saturated subsheaves. A coherent subsheaf $S$ of a torsion-free sheaf $F$ is a \emph{saturated subsheaf} if both $S$ and the quotient $Q=F/S$ are torsion-free.
\begin{lem}[cf. \cite{Kob}, \S V, Prop. 7.6]
A torsion-free coherent sheaf $F$ is semi-stable if for any saturated subsheaf $S$ of $F$, there holds $\mu(S)\le\mu(F)$, and is stable if for any saturated subsheaf $S$ of $F$ with strictly smaller rank, there holds $\mu(S)<\mu(F)$.
\end{lem}

\noindent
If $(E,\bar{\partial},H)$ is an Hermitian holomorphic vector bundle over $X$, we may study its saturated subsheaves in a more analytic manner. A \emph{$W^{1,2}$-subbundle} of $E$ is a self-adjoint section $\pi\in W^{1,2}(X,\End(E))$ satisfying $\pi^2=\pi$ and $(1-\pi)\bar{\partial}(\pi)=0$. To any saturated subsheaf $S$ of $E$, we may associate the $H$-orthogonal projection $\pi$ onto $S$, well-defined outside codimension $2$. Note that we are using the fact that any torsion-free coherent sheaf $F$ is locally free outside an analytic subset $Z(F)\subset X$ of codimension at least $2$ (\cite[\S V, Cor. 5.15]{Kob}). It follows from \cite[\S4]{Uhl-Yau} that $\pi$ is a $W^{1,2}$-subbundle of $E$.

Now, assume that $E$ is only a torsion-free coherent sheaf. A \emph{maximal semi-stable subsheaf} of $E$ is a saturated subsheaf $E_1$ such that, for every coherent subsheaf $F$ of $E$,
\begin{enumerate}[(i)]
	\item $\mu(F)\le \mu(E_1)$, and
	\item $\rank(F)\le\rank(E_1)$ if $\mu(F)=\mu(E_1)$.
\end{enumerate}
It follows that $E_1$ is semi-stable. \cite[\S V, Lem. 7.17]{Kob} shows that $E_1$ exists and is unique. By successively taking maximal semi-stable subsheaves, we obtain

\begin{prop}[cf. \cite{Kob}, \S V, Thm. 7.15]
\label{prop-hn-filtration}
Given a torsion-free coherent sheaf $E$, there exists a unique filtration by subsheaves
$$
0=\bb F_0^{\hn}(E)\subset \bb F_1^{\hn}(E)\subset\cdots\subset\bb F_{l-1}^{\hn}(E)\subset\bb F_l^{\hn}(E)=E
$$
such that, for $1\le i\le l-1$, the quotient $E/\bb F_i^{\hn}(E)$ is torsion-free and $\bb F_i^{\hn}(E)/\bb F_{i-1}^{\hn}(E)$ is the maximal semi-stable subsheaf of $E/\bb F_{i-1}^{\hn}(E)$.
\end{prop}

\noindent
This filtration is called the \emph{Harder-Narasimhan filtration} (abbreviated as \emph{HN filtration}). Note that each $\bb F^{hn}_i(E)$ is a saturated subsheaf of $E$, and that the initial subsheaf $\bb F_1^{\hn}(E)$ is the unique subsheaf of $E$ with this degree and rank.

\begin{lem}
Given a torsion-free coherent sheaf $E$, let $Q_i=\bb F_i^{hn}(E)/\bb F_{i-1}^{hn}(E)$ be the $i$th quotient in the HN filtration. Then $\mu(Q_i)>\mu(Q_{i+1})$ for all $i$.
\end{lem}
\begin{proof}
Since $\bb F_i^{\hn}(E)/\bb F_{i-1}^{\hn}(E)$ is the maximal semi-stable subsheaf of $E/\bb F_{i-1}^{\hn}(E)$, there holds
$$
\mu(\bb F_i^{\hn}(E)/\bb F_{i-1}^{\hn}(E))>\mu(\bb F_{i+1}^{\hn}(E)/\bb F_{i-1}^{\hn}(E))
$$
From the sheaf exact sequence
$$
0\rightarrow\bb F_i^{\hn}(E)/\bb F_{i-1}^{\hn}(E)\rightarrow\bb F_{i+1}^{\hn}(E)/\bb F_{i-1}^{\hn}(E)\rightarrow\bb F_{i+1}^{\hn}(E)/\bb F_i^{\hn}(E)\rightarrow 0
$$
we deduce that $\mu(Q_i)>\mu(\bb F_{i+1}^{\hn}(E)/\bb F_{i-1}^{\hn}(E))>\mu(Q_{i+1})$.
\end{proof}

\noindent
The \emph{Harder-Narasimhan type} (abbreviated as \emph{HN type}) of a holomorphic vector bundle $(E,\bar{\partial})$ of rank $r$ is the $r$-tuple given by
\begin{equation}
\label{eq-def-hn-type}
\mathrm{HN}(E,\bar{\partial}):=(\underbrace{\mu(Q_1),\cdots,\mu(Q_1)}_\text{$\rank(Q_1)$ times},\cdots,\underbrace{\mu(Q_l),\cdots,\mu(Q_l)}_\text{$\rank(Q_l)$ times})
\end{equation}
where $Q_i:=\bb F_i^{\hn}(E)/\bb F_{i-1}^{\hn}(E)$ and $\{\bb F_i^{\hn}(E)\}_{1\le i\le l}$ is the HN filtration of $E$.

In the case where $E$ is a holomorphic vector bundle equipped with an Hermitian metric $H$, for any increasing filtration $\mathcal{F}=\{F_i\}_{i=1}^l$ of $E$ by saturated subsheaves, corresponding to $W^{1,2}$-subbundles $\{\pi_i\}_{i=1}^l$, and any $l$-tuple of real numbers $(\mu_1,\cdots,\mu_l)$, we define a bounded $W^{1,2}$-endomorphism of $E$ by $\Psi(\mathcal{F},(\mu_1,\cdots,\mu_l),H)=\sum_{i=1}^l\mu_i(\pi_i-\pi_{i-1})$. The \emph{Harder-Narasimhan projection} $\Psi^{\hn}(\bar{\partial},H)$ is the endomorphism defined as above for $\mathcal{F}=\{\bb F_i^{\hn}(E)\}_{i=1}^l$ and $\mu_i=\mu\left(\bb F_i^{\hn}(E)/\bb F_{i-1}^{\hn}(E)\right)$. Following \cite{Das-Wen}, for any $\delta>0$ and $1\le p\le\infty$, we define an \emph{$L^p$-$\delta$-approximate critical Hermitian structure} on $E$ to be a smooth metric $H$ such that
\begin{equation}
\|i\Lambda F_{(\bar{\partial},H)}-\Psi^{\hn}(\bar{\partial},H)\|_{L^p}\le\delta
\end{equation}

Any semi-stable sheaf $E$ admits a Seshadri filtration, whose successive quotients are stable and has the same slope as $E$ (\cite[V. Thm. 7.18]{Kob}). In contrast to the HN filtration, the Seshadri filtration is not unique, although its associated graded object is. Putting together the Harder-Narasimhan filtration and the Seshadri filtration, we have
\begin{prop}
Given a torsion-free coherent sheaf $E$, there exists a double filtration $\{\bb F^{\hns}_{i,j}(E)\}$ with the following properties:
\begin{equation*}
\bb F^{\hn}_{i-1}(E)=\bb F^{\hns}_{i,0}(E)\subset \bb F^{\hns}_{i,1}(E)\subset\cdots\subset \bb F^{\hns}_{i,l_i}(E)=\bb F^{\hn}_i(E)
\end{equation*}
and the successive quotients $Q_{i,j}=\bb F^{\hns}_{i,j}(E)/\bb F^{\hns}_{i,j-1}(E)$ are stable, torsion-free sheaves. Moreover, $\mu(Q_{i,j})=\mu(Q_{i,j+1})$ and $\mu(Q_{i,j})>\mu(Q_{i+1,j})$ for all $i$ and $j$. The associated graded object
$$
\Gr^{\hns}(E)=\bigoplus_{i=1}^l\bigoplus_{j=1}^{l_i}Q_{i,j}
$$
is uniquely determined by the isomorphism class of $E$.\qed
\end{prop}

We call such a double filtration a \emph{Harder-Narasimhan-Seshadri filtration} of $E$ (abbreviated as \emph{HNS filtration}). Sometimes it is easier to view the HNS-filtration as a single filtration by subsheaves
\begin{equation*}
0=\bb F_0^{\hns}(E)\subset \bb F_1^{\hns}(E)\subset\cdots\subset \bb F_{l-1}^{\hns}(E)\subset \bb F_l^{\hns}(E)=E
\end{equation*}
with stable, torsion-free quotients $\bb F_i^{\hns}(E)/\bb F_{i-1}^{\hns}(E)$. The associated graded object $\Gr^{\hns}(E)$ is a torsion-free coherent sheaf, thus locally free outside codimension $2$ (\cite[\S V, Cor. 5.15]{Kob}). Now assume that $(X,g)$ is a compact Gauduchon surface. Then the singularity set
$$
Z^{\alg}:=\{x\in X:\text{$\Gr^{\hns}(E)$ is not locally free at $x$}\}
$$
is a finite set of points. Furthermore, the reflexified object $\Gr^{\hns}(E)^{**}$ is locally free outside codimension $3$ (\cite[\S V, Cor. 5.20]{Kob}), thus a holomorphic vector bundle over $X$.

\medskip

\subsection{}
The following result is an essential ingredient in identifying the Harder-Narasimhan type of Uhlenbeck limits. For the Yang-Mills flow, this result is contained in the proof of \cite[Lem. 4.3]{Das-Wen}. For any $\alpha\ge 1$, $N\in\mathbb{R}$, and $\delta>0$, define $\mathcal{H}_{\alpha,N}^{\delta}$ to be the space of smooth Hermitian metrics $H$ on $E$ such that there exists some $T\ge 0$ with
$$
\HYM_{\alpha,N}(D_t^H)<\HYM_{\alpha,N}(\vec{\mu}_0)+\delta,\quad\forall t\ge T
$$
where $D_t^H$ denote the solution to the flow $\eqref{eq-uym}$ with initial condition $(\bar{\partial},H)$.
\begin{prop}
\label{prop-hermitian-metric}
There exists some $\alpha_0>1$ such that for any $N\in\mathbb{R}$ and $\delta>0$, the space $\mathcal{H}_{\alpha,N}^{\delta}$ consists of all smooth Hermitian metrics on $E$, provided that $1\le\alpha<\alpha_0$.
\end{prop}
\begin{proof}
It suffices to prove that $\mathcal{H}_{\alpha,N}^{\delta}$ is open, closed, and nonempty. The proof for openness and closedness follows exactly as in \cite[proof of Lem. 4.3]{Das-Wen}. To prove that $\mathcal{H}_{\alpha,N}^{\delta}$ is nonempty for sufficiently small $\alpha$, the authors of \cite{Das-Wen} considered a sequence of blowing-ups $\pi:\hat{X}\rightarrow X$ such that the singularities of the HN filtration of $E$ are resolved on $\hat{X}$. Then they constructed an $L^{\infty}$-$\delta$-approximate critical Hermitian structure for the regularized filtration (\cite[Prop. 3.13]{Das-Wen}), and used it to explicitly construct an Hermitian metric $H$ on $X$ that lies in $\mathcal{H}_{\alpha,N}^{\delta}$ (\cite[Lem. 4.2]{Das-Wen}). Their proof carries over to our case, except that in constructing the $L^{\infty}$-$\delta$-approximate critical Hermitian structure for the regularized filtration, they used the following fact (cf. \cite[proof of Thm. 3.10]{Das-Wen}):

\emph{If $E$ is a semi-stable vector bundle over a compact K\"{a}hler surface $X$, then there is an $L^{\infty}$-$\delta$-approximate critical Hermitian structure on $E$.}

\noindent
The corresponding result on Gauduchon surfaces has not been established. However, it is known that if $E$ is a stable vector bundle over a compact Gauduchon surface $X$, then $E$ admits an Hermitian-Einstein metric (cf. \cite{Lub}). Therefore, we may argue as in \cite[Thm. 3.10]{Das-Wen} for the HNS filtration, and the $L^{\infty}$-$\delta$-approximate critical Hermitian structure exists by the proof of \cite[Prop. 3.13]{Das-Wen}, applied to the HNS filtration.
\end{proof}

\noindent
Before we prove our main theorems, we collect the following convergence results from \cite{Das-Wen} and list them as lemmas. Their proofs carry over verbatim to the case of Gauduchon surfaces.

\begin{lem}[cf. \cite{Das-Wen}, Prop. 2.21]
\label{lem-dw-1}
Let $D_j$ be a sequence of complex gauge equivalent integrable connections on an Hermitian vector bundle $(E,H_0)$ of rank $r$, such that for some $p>4$, $D_j$ converges weakly in $W^{1,p}$ to an Uhlenbeck limit $D_{\infty}$. Assume furthermore that $\Lambda F_{D_j}\rightarrow\Lambda F_{D_{\infty}}$ in $L^1$. Then $\mathrm{HN}(E,D_0'')\le\mathrm{HN}(E_{\infty},D_{\infty}'')$.
\end{lem}

\begin{lem}[cf. \cite{Das-Wen}, Lem. 4.5(1)]
\label{lem-dw-2}
Let $D_j$ be as in Lem. \ref{lem-dw-1}, and assume furthermore that $\mathrm{HN}(E,D_0'')=\mathrm{HN}(E_{\infty},D_{\infty}'')$ and $\|\Lambda F_{D_j}\|_{L^{\infty}}$ is uniformly bounded. Let $\{\pi_j^{(i)}\}$ be the HN filtration of $(E,D_j'')$ and $\{\pi_{\infty}^{(i)}\}$ the HN filtration of $(E_{\infty},D_{\infty}'')$. Then after passing to a subsequence, $\pi_j^{(i)}\rightarrow\pi_{\infty}^{(i)}$ strongly in $L^p\cap W^{1,2}_{\loc}$, for all $1\le p<\infty$ and all $i$.
\end{lem}

\begin{lem}[cf. \cite{Das-Wen}, Thm. 5.1]
\label{lem-dw-3}
Suppose $D_j$ is a sequence of complex gauge equivalent integrable connections, and let $\vec{\mu}_0=\mathrm{HN}(E,D_0'')$. Assume $\|F_{D_j}\|_{L^2}$ and $\|\Lambda F_{D_j}\|_{L^{\infty}}$ are uniformly bounded, and $\HYM(D_j)\rightarrow\HYM(\vec{\mu}_0)$. Then there is a connection $D_{\infty}$ on an Hermitian vector bundle $E_{\infty}$, and a finite set $Z^{\an}\subset X$ such that
\begin{enumerate}[(i)]
	\item $(E_{\infty},D_{\infty}'')$ is holomorphically isomorphic to $\Gr^{\hns}(E,D_0'')^{**}$;
	
	\item $E$ and $E_{\infty}$ are identified outside $Z^{\an}$ via $W^{2,p}_{\loc}$-isometries for all $p$;
	
	\item Via the isometries in (ii), and after passing to a subsequence, $D_j\rightarrow D_{\infty}$ in $L^2_{\loc}$ away from $Z^{\an}$.
\end{enumerate}
\end{lem}

\noindent
As in \cite[proof of Thm. 4.1]{Das-Wen}, we may combine Prop. \ref{prop-hermitian-metric} and Lem. \ref{lem-dw-1} to get
\begin{prop}
\label{prop-uhlenbeck-limit-type}
Let $D_t$ be a solution to \eqref{eq-uym}, and $D_{\infty}$ be an Uhlenbeck limit along any minimizing sequence $t_j\rightarrow\infty$, on vector bundle $E_{\infty}$. Then $\mathrm{HN}(E,D_0'')=\mathrm{HN}(E_{\infty},D_{\infty}'')$.\qed
\end{prop}

\noindent
The proof of our main theorems now follow trivially from prior estimates along the flow \eqref{eq-uym} and the above convergence results of \cite{Das-Wen}.

\begin{thm}
Let $(E,\bar{\partial})$ be a holomorphic vector bundle over a compact Gauduchon surface $(X,g)$. Given any $\delta>0$ and any $1\le p<\infty$, there is an $L^p$-$\delta$-approximate critical Hermitian structure on $E$.
\end{thm}
\begin{proof}
Argue as in \cite[proof of Thm. 3.11]{Das-Wen}. Let $D_t$ be the solution to \eqref{eq-uym} with initial condition $D_0''=\bar{\partial}$, and let $D_{\infty}$ be an Uhlenbeck limit along some minimizing sequence $t_j\rightarrow\infty$. By Prop. \ref{prop-uhlenbeck-limit-type}, $\mathrm{HN}(E,D_0'')=\mathrm{HN}(E_{\infty},D_{\infty}'')$ and we may apply Lem. \ref{lem-dw-2} to obtain $\Psi^{\hn}(D_{t_j}'',H)\rightarrow \Psi^{\hn}(D_{\infty}'',H_{\infty})$ in $L^p$. Since $H_{\infty}$ is a critical Hermitian structure, $i\Lambda F_{D_{\infty}}=\Psi^{\hn}(D_{\infty}'',H_{\infty})$. Hence
$$
\|i\Lambda F_{D_{t_j}}-\Psi^{\hn}(D_{t_j}'',H)\|_{L^p}\le\|\Lambda F_{D_{t_j}}-\Lambda F_{D_{\infty}}\|_{L^p}+\|\Psi^{\hn}(D_{t_j}'',H)-\Psi^{\hn}(D_{\infty}'',H)\|_{L^p}\rightarrow 0
$$
using \eqref{eq-lambda-f-lp-limit} for the first term on the right hand side.
\end{proof}

\begin{thm}
Let $(E,\bar{\partial})$ be an Hermitian holomorphic vector bundle over a compact Gauduchon surface $(X,g)$. Let $D_t$ be the solution to \eqref{eq-uym} with initial condition $D_0''=\bar{\partial}$. Given any sequence $t_j\rightarrow\infty$, there exists a connection $D_{\infty}$ on an Hermitian vector bundle $E_{\infty}$, and a finite set $Z^{\an}\subset X$ such that
\begin{enumerate}[(i)]
	\item $(E_{\infty},D_{\infty}'')$ is holomorphically isomorphic to $\Gr^{\hns}(E,D_0'')^{**}$;
	
	\item $E$ and $E_{\infty}$ are identified outside $Z^{\an}$ via $W^{2,p}_{\loc}$-isometries for all $p$;
	
	\item Via the isometries in (ii), and after passing to a subsequence, $D_j\rightarrow D_{\infty}$ in $L^2_{\loc}$ away from $Z^{\an}$.
\end{enumerate}
\end{thm}
\begin{proof}
We first take a minimizing sequence $\tilde{t}_j\rightarrow\infty$, and obtain an Uhlenbeck limit $\tilde{D}_{\infty}$ on some Hermitian vector bundle $\tilde{E}_{\infty}$. Since $i\Lambda F_{\tilde{D}_{\infty}}$ has constant eigenvalues given by the $r$-tuple $\mathrm{HN}(\tilde{E}_{\infty},\tilde{D}_{\infty}'')$, Prop. \ref{prop-uhlenbeck-limit-type} implies $\HYM(\tilde{D}_{\infty})=\HYM(\vec{\mu}_0)$, where $\vec{\mu}_0=\mathrm{HN}(E,\bar{\partial})$. Apply Prop. \ref{prop-limit-other-hym} for $\alpha=2$ and $N=0$, and we obtain $\lim_{t\rightarrow{\infty}}\HYM(D_t)=\HYM(\tilde{D}_{\infty})$. So $\lim_{t\rightarrow\infty}\HYM(D_t)=\HYM(\vec{\mu}_0)$. The theorem then follows from Lem. \ref{lem-dw-3}, applied to any given sequence $t_j\rightarrow\infty$.
\end{proof}

\medskip

\section{Appendix}
\subsection*{Follow-up on Rmk. \ref{rem-surface-torsion}}
In this section, we will first prove
\begin{prop}
\label{prop-general-torsion}
Let $(X,g)$ be an Hermitian manifold with fundamental 1-form $\omega$. Then for any $\xi\in\Lambda^2X$, there holds
\begin{equation}
\label{eq-general-torsion}
(\tau+\bar{\tau})^*\xi=-\frac{*\left(d(\omega^{n-2})\wedge\mathbf{I}(\xi_2)\right)}{(n-2)!}+\frac{2*\left(d(\omega^{n-1})\cdot\xi_0\right)}{(n-1)!}
\end{equation}
where $\xi=\xi_2+L\xi_0$ is the Lefschetz decomposition of $\xi$, and $\mathbf{I}=\sum_{p,q}i^{p-q}\Pi_{p,q}$ is the Weil operator.
\end{prop}

\noindent
To prove Prop. \ref{prop-general-torsion}, we need a standard result:
\begin{lem}
For any primitive form $\alpha\in\Lambda^kX$, one has
\begin{equation}
\label{eq-huybrechts}
*L^j\alpha=(-1)^{\frac{k(k+1)}{2}}\frac{j!}{(n-k-j)!}\cdot L^{n-k-j}\mathbf{I}(\alpha)
\end{equation}
\end{lem}
\begin{proof}
See \cite[Prop. 1.2.31]{Huy} for example.
\end{proof}

\begin{proof}[Proof of Prop. \ref{prop-general-torsion}]
As in the proof of Lem. \ref{lem-surface-torsion}, we have $(\tau+\bar\tau)^*\xi=*(d\omega\wedge *L\xi)$. The case $n=2$ is exactly given by Lem. \ref{lem-surface-torsion}, so we assume $n\ge 3$ in what follows.

Consider the Lefschetz decomposition $\xi=\xi_2+L\xi_0$. We may then compute using \eqref{eq-huybrechts}:
\begin{align*}
(\tau+\bar{\tau})^*\xi=&*\left(d\omega\wedge *L\xi_2\right)+*\left(d\omega\wedge *L^2\xi_0\right)\\
=&-\frac{*\left(d\omega\wedge\omega^{n-3}\wedge\mathbf{I}(\xi_2)\right)}{(n-3)!}+\frac{2*\left(d\omega\wedge \omega^{n-2}\cdot\xi_0\right)}{(n-2)!}
\end{align*}
since $\mathbf{I}(\xi_0)=\xi_0$. The desired identity \eqref{eq-general-torsion} then follows from $d(\omega^{n-1})=(n-1)d\omega\wedge\omega^{n-2}$, and $d(\omega^{n-2})=(n-2)d\omega\wedge\omega^{n-3}$.
\end{proof}

\begin{cor}
Suppose $n\ge 3$ and $d(\omega^{n-2})=0$. Then $(\tau+\bar{\tau})^*\xi=0$ for all $\xi\in\Lambda^2X$.
\end{cor}
\begin{proof}
Observe that
\begin{equation}
\label{eq-n-2}
d(\omega^{n-2})=0\implies d(\omega^{n-1})=0,\quad\text{for}\quad n\ge 3
\end{equation}
Indeed, this follows from $d(\omega^{n-1})=(n-1)d\omega\wedge\omega^{n-2}$ and $d(\omega^{n-1})=d\omega\wedge\omega^{n-2}+\omega\wedge d(\omega^{n-2})$. The corollary now follows from \eqref{eq-general-torsion} and \eqref{eq-n-2}.
\end{proof}

\medskip

\subsection*{Follow-up on Rmk. \ref{rem-defn-uhlenbeck-limit}}
Note that the condition
\begin{enumerate}[(i)]
	\item[(iii${}^*$)] for any compact set $K\subset\subset X-Z^{\an}$, a $W^{2,p}$-isometry $\tau^K:(E_{\infty},H_{\infty})|_K\rightarrow(E,H)|_K$ such that for $K\subset K'\subset\subset X-Z^{\an}$, $\tau^K=\tau^{K'}|_K$, and $\tau^K(D_{j_k})\rightharpoonup D_{\infty}$ weakly in $W^{1,p}(K)$.
\end{enumerate}
is equivalent to the existence of some $\tau\in\Aut^{2,p}_{\loc}(E|_{X-Z^{\an}},H)$ such that $\tau(D_{j_k})\rightharpoonup D_{\infty}$ weakly in $W^{1,p}_{\loc}$. To show that (ii)(iii) imply (ii${}^*$)(iii${}^*$) possibly after passing to a subsequence, note first that $X-Z^{\an}=\bigcup_{l\in\mathbb{N}}X_l$ is exhausted by countably many compact submanifolds $X_l$. Let $\tau_{k}^{(0)}$ denote the original sequence $\tau_{j_k}$. Hence it suffices to find, for each $X_l$, a subsequence of $\tau_{k}^{(l-1)}$ denoted by $\tau_{k}^{(l)}$, such that
\begin{enumerate}[(i)]
	\item $\tau_k^{(l)}\rightharpoonup\tau^{(l)}$ weakly in $W^{2,p}$, for some $\tau^{(l)}\in\Aut^{2,p}(E|_{X_l},H)$, and
	
	\item $\tau^{(l)}(D_k)\rightharpoonup D_{\infty}$ weakly in $W^{1,p}$, as connections on $E|_{X_l}$,
\end{enumerate}
and then take the diagonal sequence over $l$. This subsequence can be found using

\begin{lem}
Suppose $(X,g)$ is a compact Hermitian manifold of dimension $n$, and $(E,H)$ an Hermitian vector bundle over $X$. Suppose $D_k\in\mathcal{A}^{1,p}(E,H)$ and $\tau_k\in\Aut^{2,p}(E,H)$, where $p>\frac{n}{2}$, such that
\begin{enumerate}[(i)]
	\item $\{D_k\}$ is uniformly bounded in $W^{1,p}$, and
	
	\item $\tau_k(D_k)\rightharpoonup D_{\infty}$ weakly in $W^{1,p}$ for some $D_{\infty}\in\mathcal{A}^{1,p}(E,H)$.
\end{enumerate}
Then after passing to a subsequence, $\tau_k\rightharpoonup\tau$ weakly in $W^{2,p}$ for some $\tau\in\Aut^{2,p}(E,H)$ and $\tau(D_k)\rightharpoonup D_{\infty}$ weakly in $W^{1,p}$.
\end{lem}
\begin{proof}
We may cover $X$ by finitely many bundle charts, and work over one such chart $U$, where we may write $D_k=d+A_k$ and $D_{\infty}=d+A_{\infty}$. Note that for each $\tau\in\Aut^{2,p}(E|_U,H)$ and a $W^{1,p}$-1-form $A$ over $U$, there holds
\begin{equation}
\label{eq-bound-gauge-transform}
\|\tau A\tau^{-1}\|_{W^{1,p}}\le C\|A\|_{W^{1,p}}(1+\|d\tau(\tau^{-1})\|_{L^{2p}})
\end{equation}
by the Sobolev embedding $W^{1,p}\hookrightarrow L^{2p}$. Therefore, the expression
$$
\tau_k(A_k)=\tau_kA_k\tau_k^{-1}-d\tau_k(\tau_k^{-1})
$$
implies that $\|d\tau_k(\tau_k^{-1})\|_{L^{2p}}$ is uniformly bounded, and thus so is
$$
\|d\tau_k(\tau_k)^{-1}\|_{W^{1,p}}\le\|\tau_k(A_k)\|_{W^{1,p}}+\|\tau_kA_k\tau_k^{-1}\|_{W^{1,p}}
$$
by the hypothesis (ii) and \eqref{eq-bound-gauge-transform}. Using the estimate in \cite[Lem. B.5]{Weh}, the bound on $\|d\tau_k(\tau_k)^{-1}\|_{W^{1,p}}$ implies that $\tau_k$ is uniformly bounded in $W^{2,p}$. Let $\tau$ be a weak $W^{2,p}$-limit of $\tau_k$ along some subsequence. We now write
$$
\tau(A_k)=\tau A_k\tau^{-1}-d\tau(\tau^{-1})
$$
and another application of \eqref{eq-bound-gauge-transform} implies that $\|\tau A_k\tau^{-1}\|_{W^{1,p}}$ is uniformly bounded, and thus $\tau(A_k)$ has a weak limit in $W^{1,p}$ along some further subsequence, say $\tau(A_k)\rightharpoonup A_{\infty}'$ weakly in $W^{1,p}$. Choose $p'$ with $\frac{n}{2}<p'<p$. Then $\tau(A_k)\rightarrow A_{\infty}'$, $\tau_k(A_k)\rightarrow A_{\infty}$ strongly in $W^{1,p'}$, and $\tau_k\rightarrow\tau$ strongly in $W^{2,p'}$. By continuity of composition, inversion of $W^{2,p'}$-gauge transformations, and their action on $W^{1,p'}$-connections (cf. \cite[Lem. A.5, A.6]{Weh}),
$$
\tau(A_k)=(\tau\circ\tau_k^{-1})\left(\tau_k(A_k)\right)\rightarrow A_{\infty}\quad\text{strongly in $W^{1,p'}$}
$$
Therefore $A_{\infty}'=A_{\infty}$ as $W^{1,p}$-connections.
\end{proof}

\end{document}